\documentclass{article}

\pdfoutput=1

\usepackage{amsmath, amssymb, amsthm, enumerate, fullpage}
\usepackage{verbatim, enumitem, bbm, etoolbox}
\usepackage{mathtools}
\usepackage{tensor}
\usepackage{leftidx}
\usepackage{cancel}

\newcommand{\tp}{\textrm{tp}}

\newcommand{\icl}{\textrm{icl}}

\newcommand{\z}{\overline}
\newcommand{\f}{\mathfrak}
\renewcommand{\u}{\underline}

\theoremstyle{plain}
\newtheorem{thm}{Theorem}
\newtheorem{theorem}[thm]{Theorem}
\newtheorem{lemma}[thm]{Lemma}
\newtheorem{cor}[thm]{Corollary}

\newtheorem{prop}[thm]{Proposition}
\newtheorem{fact}[thm]{Fact}

\numberwithin{thm}{section}

\numberwithin{subcase}{case}

\theoremstyle{definition}
\newtheorem{defn}[thm]{Definition}
\newtheorem{remark}[thm]{Remark}

\newtheorem{notation}[thm]{Notation}
\newtheorem*{theorem*}{Theorem}

\def\dnf{\mathrel{\raise0.2ex\hbox{\ooalign{\hidewidth
$\vert$\hidewidth\cr\raise-0.9ex\hbox{$\smile$}}}}}

\def\forks{\mathrel{\raise0.2ex\hbox{\ooalign{\hidewidth
$\cancel\vert$\hidewidth\cr\raise-0.9ex\hbox{$\smile$}}}}}

\def\dindep{\mathrel{\raise0.2ex\hbox{\ooalign{\hidewidth
$\vert\rlap{${}^d$}$\hidewidth\cr\raise-0.9ex\hbox{$\smile$}}}}}

\def\ndindep{\mathrel{\raise0.2ex\hbox{\ooalign{\hidewidth
$\cancel\vert\rlap{${}^d$}$\hidewidth\cr\raise-0.9ex\hbox{$\smile$}}}}}

\newcommand{\proves}{\vdash}

\newcommand{\finsubset}{\subseteq_\textrm{Fin}}

\newcommand{\Kfin}{K_{\z{\alpha}}}

\newcommand{\Sfin}{S_{\z{\alpha}}}

\begin{document}
\title{The theories of Baldwin-Shi hypergraphs and their atomic models}

\author{Danul K. Gunatilleka\thanks{Partially supported
		by NSF grant DMS-1308546.}}
\date{\vspace{-5ex}}

\newcommand{\Addresses}{{% additional braces for segregating \footnotesize
		\bigskip
		\footnotesize
		
		\textsc{Department of Mathematics, University of Maryland at College Park}\par\nopagebreak
		\textit{E-mail address}: \texttt{danulg@math.umd.edu}
		
		%\medskip
		
		%M.~Dane (Corresponding author), \textsc{Atmospheric Research Station,
		%	Pala Lundi, Fiji}\par\nopagebreak
		%\textit{E-mail address}, M.~Dane: \texttt{DaneMark@@ffr.choice}
		
		%\medskip
		
		%J.~Jones, \textsc{Department of Philosophy, Freedman College,
		%	Periwinkle, Colorado 84320}\par\nopagebreak
		%\textit{E-mail address}, J.~Jones: \texttt{id739e@@oseoi44 (Bitnet)}
		
	}}
	
	\maketitle
	
	%\renewcommand\Authands{ and }
	%\date{\vspace{-5ex}}
	\maketitle
	
	\begin{abstract}
		We show that the quantifier elimination result for the Shelah-Spencer almost sure theories of sparse random graphs $G(n,n^{-\alpha})$ given by Laskowski in \cite{Las1} extends to their various analogues. The analogues will be obtained as theories of generic structures of certain classes of finite structures with a notion of strong substructure induced by rank functions and we will call the generics Baldwin-Shi hypergraphs. In the process we give a method of constructing extensions whose `relative rank' is negative but arbitrarily small in context. We give a necessary and sufficient condition for the theory of a Baldwin-Shi hypergraph to have atomic models. We further show that for certain well behaved classes of theories of Baldwin-Shi hypergraphs, the existentially closed models and the atomic models correspond. 
	\end{abstract}

	\section{Introduction}\label{sec:intro}
	Fix a finite relational language $L$ where each relation symbol has arity at least $2$ and let $K_L$ be the class of finite structures where each relation symbols is interpreted reflexively and symmetrically. Fix a function $\z\alpha:L\rightarrow (0,1]$ \textit{with} the additional restriction that if \textit{all} the relation symbols are $2$-ary then it is not the case $\z\alpha(E)=1$ for each $E\in L$. Define a \textit{rank function} $\delta:K_L\rightarrow \mathbb{R}$ by $\delta(\f{A})=|A|-\sum_{E\in L}\z\alpha(E)|E^{\f{A}}|$ where $|E^{\f{A}}|$ is the number of subsets of $A$ on which $E$ holds. Let $\Kfin=\{\f{A}\in K_L:\delta(\f{A'})\geq 0 \text{ for all }\f{A'\subseteq A}\}$. Given $\f{A,B}\in\Kfin$, we say that $\f{A\leq B}$ if and only if $\f{A\subseteq B}$ and $\delta(\f{A})\leq \delta(\f{A'})$ for all $\f{A\subseteq A'\subseteq B}$. The class $(\Kfin,\leq)$ forms a Fra\"{i}ss\'{e} class, i.e. $\Kfin$ has amalgamation and joint embedding under $\leq$. In \cite{BdSh}, Baldwin and Shi initiated a systematic study of the \textit{generic structures} constructed from various sub-classes $K^*\subseteq \Kfin$ where $(K^*,\leq)$ forms a Fra\"{i}ss\'{e} class. In particular they obtained the stability of the theory of the generic for $(\Kfin,\leq)$. We call the generic for $(\Kfin,\leq)$ the \textit{Baldwin-Shi hypergraph for $\z\alpha$}. 
	
	In \cite{BdShelah1}, Baldwin and Shelah showed that the results regarding almost sure theories of \textit{graphs} studied by Shelah and Spencer in \cite{SpencerShelah} extended to their natural \textit{hypergraph} counterparts. They further connected two disparate lines of research when they showed that these almost sure theories corresponded to certain theories of Baldwin-Shi hypergraphs, allowing us to establish the Baldwin-Shi hypergraphs as analogues of the almost sure theories. These results of \cite{BdShelah1} hinge on a $\forall\exists\forall$-axiomatization of the resulting theory. Assuming that the values of $\z\alpha(E)$ as $E$ ranges through $L$ is linearly independent over $\mathbb{Q}$, Laskowski in \cite{Las1}, provided a simpler $\forall\exists$-axiomatization of the corresponding theories of Baldwin-Shi hypergraphs. He also obtained quantifier elimination result down to \textit{chain minimal formulas} (see definition \ref{defn:ExtensionFormula})  Later, in \cite{IkKiTs} Ikeda, Kikyo and Tsuboi showed that the $\forall\exists$-axiomatization, denoted by $\Sfin$, holds for \textit{all} theories of Baldwin-Shi hypergraphs. However their methods did not establish a quantifier elimination result in the spirit of Laskowski.  
	
	In this paper we begin by extending Laskowski's quantifier elimination for all $\Sfin$. We then isolate properties of $\z\alpha$; \textit{coherence}, i.e. the linear dependence of $\{\z\alpha(E):E\in L\}$, and \textit{rationality}, i.e. $\z\alpha(E)$ is rational for all $E\in L$, that play a role in determining properties of $\Sfin$. We show that coherence is a necessary and sufficient condition for the existence of atomic models and that rationality is a necessary and sufficient condition to guarantee that the atomic and existentially closed models correspond.   
	
	We begin in Section \ref{sec:prelim} by introducing preliminary notions that we will be using throughout this paper. In Section \ref{sec:ExistThm} we deal primarily with finite structures. One of the key results is Theorem \ref{thm:OmitPrelim2}, which yields the existence of certain finite structures over some fixed finite structure that witness a very small drop in rank. This theorem plays a central role in many results throughout this paper and in \cite{DG2}. Another key result is Theorem \ref{thm:AnnhilConstruc} which establishes the existence of rank $0$ extensions of finite substructures given coherence of $\z\alpha$. 
	
	The key result of Section \ref{sec:QuantElim}, which is mainly aimed at generalizing the results of \cite{Las1}, is Theorem \ref{thm:QuantElim}. It states that $\Sfin$ admits quantifier elimination down to the level of \textit{chain minimal extension formulas}. It also yields the completeness of $\Sfin$ and a characterization of algebraically closed sets for $\Sfin$ as stated in Theorem \ref{lem:ClosEqui}. We end the section with some basic facts about types over (algebraically) closed sets that will be useful throughout. 
	
	Section \ref{sec:AtomECMod} is devoted to a study of the atomic models of $\Sfin$. In Theorem \ref{thm:ExistenceOfAtomic} of Section \ref{sec:AtomECMod} we establish that coherence of $\z\alpha$ is a necessary and sufficient condition for the corresponding $\Sfin$ to have atomic models. In Theorem \ref{thm:RationalAlphaAndCoherence} we show that rationality of $\z\alpha$ is equivalent to every model of the $\Sfin$ being isomorphically embeddable in an atomic model of $\Sfin$. We end with Appendix A which contains a collection of well known number theoretic results that is used throughout.  
	
	The author would like to thank Chris Laskowski for all his help and guidance in the preparation of this paper.
	
\section{Preliminaries}\label{sec:prelim}

We work throughout with a finite relational language $L$ where \textit{each relation symbol $E\in L$ is at least binary}. Let $ar:L\rightarrow \{n:n\in\omega \text{ and } n\geq 2\}$ be a function that takes each relation symbol to its arity. This section is devoted to introducing notation, definitions and some facts about the rank function $\delta$ (see Definition \ref{defn:ClassKfin}) that will be useful throughout. The results in this section are well known or follow from routine calculations involving $\delta$.

\subsection{Some general notions}\label{subsec:GenNotions}

We begin with some notation.

\begin{notation}  Fraktur letters will denote $L$-structures. Their Latin counterparts will, as we shall see, denote either the structure or the underlying set. Let $\f{Z}$ be an $L$-structure and let $X,Y\subseteq Z$. We will adapt the practice of writing $XY$ for $X\cup Y$. Since we are in a finite relational language $X, Y, XY$ will have a natural $L$-structures associated with them, i.e. the $L$-structures with universe $X, Y, XY$ that are substructures of $\f{Z}$, respectively. By a slight abuse of notation we write $X,Y, XY$ for these $L$-structures. It will be clear by context what the notation refers to. We write $X\finsubset Z$, $\f{X\finsubset Z}$ to indicate that $|X|$ is finite.  
\end{notation}

\begin{notation} We will use $\emptyset$ to denote the unique $L$-structure with no elements. Further given $L$-structures $\f{X,Y}$, there is a uniquely determined $L$-structure whose universe is $X\cap{Y}$. We denote this structure by $\f{X\cap Y}$. 
\end{notation}

\begin{notation}
	We let $K_L$ denote the class of all finite $L$ structures $\f{A}$ (including the empty structure), where each $E\in{L}$ is interpreted symmetrically and irrelexively in $A$: i.e. $\f{A}\in K_L$ if and only if for every $E\in L$, if $\f{A}\models E(\z{a})$, then  $\z{a}$ has no repetitions and $\f{A}\models E(\pi(\z{a}))$ for every permutation $\pi$ of $\{ 0,\ldots,n-1\}$. By $\z{K_L}$ we denote the class of $L$-structures whose finite substructures all lie in $K_L$, i.e. $\z{K_L}=\{\f{M}:\f{M}\text{ an } L-\text{structure and if } \f{A\finsubset M}, \text{ then }\f{A}\in K_L \}$.
\end{notation}

\begin{notation}
	Fix any $E\in{L}$. Given $\f{A}\in K_L$, $N_{E}(\f{A})$ will denote the \textit{number of distinct subsets} of $A$ on which $E$ holds positively inside of $\f{A}$. The set of such subsets will be denoted by $E^{\f{A}}$. Consider an $L$-structure whose finite substructures are all in $K_L$ and let $A,B,C\subseteq Z$ be finite. Now $N_{E}({A},{B})$ will denote the \textit{number of distinct subsets} of ${AB}$ on which $E$ holds with \textit{at least one element from} $A$ and \textit{at least one element from} $B$ inside of ${AB}$. We further let $N_{E}({A},{B},{C})$ denote the number of \textit{distinct subsets} of $A\cup{B}\cup{C}$ on which $E$ holds with \textit{at least one element} from $A$ and \textit{at least one element} from $C$. 
\end{notation}

We now introduce the class $\Kfin$ as a subclass of $K_L$.

\begin{defn}\label{defn:ClassKfin}
	Fix a function ${\z\alpha}:L\rightarrow (0,1]$ \textit{with} the property that if all of the relation symbols in $L$ have arity $2$, then it is not the case that $\z\alpha(E)=1$ for all $E\in L$.  Define a function $\delta:K_L\rightarrow \mathbb{R}$ by $\delta(\f{A})= \vert A \vert-\sum_{E\in{L}}{{\z\alpha}(E)N_{E}(\f{A})}$ for each $\f{A}\in K_L$. We let $\Kfin=\{\f{A}  \vert \delta(\f{A'})\geq{0}\text{ for all }\f{A'}\subseteq{\f{A}}\}.$	 
\end{defn}

We adopt the convention $\emptyset\in K_L$ and hence $\emptyset\in \Kfin$ as $\delta(\emptyset)=0$. It is easily observed that $\Kfin$ is closed under substructure. Further the rank function $\delta$ allows us to view both $K_L$ and $\Kfin$ as collections of weighted hypergraphs. We proceed to use the rank function to define a notion of strong substructure $\leq$. Typically the notion of $\leq$ is usually defined on $\Kfin\times \Kfin$. However, we define the concept on the broader class $K_L\times K_L$. This will allow us to make the exposition significantly simpler via Remark \ref{rmk:ConstrVerifSimpl}. 

\begin{defn}\label{defn:defStrong}
	Given $\f{A, B}\in{K_L}$ with $\f{A\subseteq B}$, we say that $\f{A}$ \textit{is strong in} $\f{B}$, denoted by $\f{A}\leq{\f{B}}$ if and only if $\f{A\subseteq{B}}$ and $\f{\delta{(A)}\leq{\delta{(A')}}}$ for all $\f{A\subseteq{A'}\subseteq{B}}$. 
\end{defn}

\begin{remark}\label{fact:ExtSmthFraClass}
	The relation $\leq$ on $K_L \times K_L$ is reflexive, transitive and has the property that given $\f{A,B,C}\in K_L$, if $\f{A\leq{C}}$, $\f{B\subseteq{C}}$ then $\f{A\cap{B}\leq{B}}$ (use $(2)$ of Fact \ref{lem:MonoRelRankOvBase}). The same statement holds true if we replace $K_L$ by $\Kfin$ in the above. Further for any given $\f{A}\in\Kfin$, $\emptyset\leq{\f{A}}$.   
\end{remark}  

\begin{remark}\label{rmk:ConstrVerifSimpl}
	Let $\f{A}\in \Kfin$, $\f{B}\in K_L$ with $\f{A\subseteq B}$. Using $(2)$ of Fact \ref{lem:MonoRelRankOvBase}, we easily obtain that if $\f{A\leq B}$, then $\f{B}\in\Kfin$. 
\end{remark}

\begin{defn}\label{defn:PostiveRankInfStructures}
	By $\z\Kfin$ we denote the class of all $L$-structures whose finite substructures are all in $\Kfin$, i.e. $\z{\Kfin}=\{\f{M}:\f{M}\text{ an } L-\text{structure and if } \f{A\finsubset M}, \text{ then }\f{A}\in \Kfin\}$. 
\end{defn}

The following definition extends the notion of strong substructure to structures in $\z{K_L}$:

\begin{defn}
	Let $\f{X}\in\z{K_L}$. For $\f{A\finsubset{X}}$, $\f{A}$ \textit{is strong in} $\f{X}$, denoted by $\f{A\leq X}$, if $\f{A\leq B}$ for all $\f{A\subseteq B \finsubset Z}$. Given $\f{A}'\in K_L$ an embedding $f:\f{A'\rightarrow{\f{X}}}$ is called a \textit{strong embedding} if $f(\f{A'})$ is strong in $\f{X}$.
\end{defn}

\begin{defn}\label{defn:BasicDel2} 
	Let $n$ be a positive integer. A set $\{\f{B}_i:i<n\}$ of elements of $\Kfin$ is disjoint over $\f{A}$ if $\f{A}\subseteq{\f{B}_i}$ for each $i<n$ and $B_i\cap{B_j}=A$ for $i<j<n$. If $\{\f{B}_i:i<n\}$ is disjoint over $\f{A}$, then $\f{D}$ is a \textit{join} of $\{\f{B}_i:i<n\}$ if the universe $D=\bigcup\{B_i:i<n\}$ and $\f{B}_i\subseteq{\f{D}}$ for all $i$. A join is called the \textit{free join}, which we denote by $\oplus_{i<n}\f{B}_i$ if there are no additional relations, i.e. $E^\f{D}=\bigcup\{E^{\f{B}_i}:i<n\}$ for all $E\in{L}$. In the case $n=2$ we will use the notation $\f{B}_0\oplus_{\f{A}}\f{B}_1$ for $\oplus_{i<2}\f{B}_i$. We note that there are obvious extension of these notions to $K_L$, $\z{K_L}$, $\z\Kfin$ and to infinitely many structures $\{\f{X}_i:i<\kappa\}$ being disjoint/joined/freely joined over some fixed $\f{Y}\subseteq\f{X}_i$ for each $i<\kappa$. 
\end{defn}

\begin{fact}\label{lem:FullAmalg1}
	If $\f{B, C}\in{\Kfin}$, $\f{A=B\cap{C}}$, and $\f{\f{A\leq{B}}}$, then $\f{B\oplus_{A}C}\in{\Kfin}$ and $\f{C\leq{B\oplus_{A}C}}$. 
\end{fact}

We now turn our attention towards constructing the generic structure for $(\Kfin,\leq)$.

\begin{defn}\label{defn:GenericStr}
	A countable structure $\f{M}\in\z{\Kfin}$ is said to be the generic for $(\Kfin,\leq)$ if \begin{enumerate}
		\item $\f{M}$ is the union of an $\omega$-chain $\f{A}_0\leq\f{A}_1\leq\ldots$ with each $\f{A}_i\in \Kfin$. 
		\item If $\f{A,B}\in\Kfin$ with $\f{A\leq B}$ and $\f{A\leq{M}}$, then there is $\f{{B'}\leq {M}}$ such that $\f{{B}\cong_{A}{B'}}$. 
	\end{enumerate}
\end{defn}

\begin{fact}\label{fact:ExistGenStr}
	$(\Kfin,\leq)$ is a Fra\"{i}ss\'{e} class (i.e. $(\Kfin,\leq)$ satisfies joint embedding and amalgamation with respect to $\leq$) and a generic structure for $(\Kfin,\leq)$ exists and is unique up to isomorphism.
\end{fact}

This justifies the following definition:

\begin{defn}
	For a fixed $\z\alpha$ we call the generic for $(\Kfin,\leq)$ the \textit{Baldwin-Shi hypergraph for $\z\alpha$}.
\end{defn}

\subsection{Closed sets}\label{subsec:ClosedSets}

In this section we generalize the notion of strong substructure to substructures of arbitrary size by introducing the notion of a closed set. This will provide us with a useful tool for analyzing the various theories of Baldwin-Shi hypergraphs. 

\begin{defn}\label{defn:MinPair}
	Let $\f{A,B}\in{K_L}$. Now $\f{(\f{A,B})}$ is a \textit{minimal pair} if and only if $\f{A\subseteq{B}}$, $\f{A\leq{C}}$ for all $\f{A\subseteq{C}\subset{B}}$ but $\f{A\nleq{B}}$.
\end{defn}

Note that  $\f{(\f{A,B})}$ is a minimal pair if and only if $\f{A\subseteq{B}}$, $\f{\delta{(A)}\leq{\delta{(C)}}}$ for all $\f{A\subseteq{C}\subset{B}}$ but $\f{\delta(\f{B})<\delta{(A)}}$. 

\begin{defn}\label{defn:ClosedSet}
	Let $\f{Z}\in{\z{K_L}}$ and ${X\subseteq{Z}}$. We say ${X}$ \textit{is closed in} $\f{Z}$ if and only if for all ${A\finsubset{X}}$, if ${(\f{A,B})}$ is a minimal pair with ${B\subseteq{Z}}$, then ${B\subseteq{X}}$.      
\end{defn}

\begin{remark}\label{rmk:Closed=StrongForFinite}
	As any $\f{A,B,C}\in K_L$ with $\f{A\leq{C}}$ and $\f{B\subseteq{C}}$ satisfies $\f{A\cap{B}\leq{B}}$ (see Remark \ref{fact:ExtSmthFraClass}) an easy argument yields that given $\f{Z}\in K_L$ and $\f{A\finsubset Z}$, $\f{A\leq Z}$ if and only if $\f{A}$ is closed in $\f{Z}$.  
\end{remark}

It is immediate from the above definition that any $\f{Z}\in{\z{K_L}}$, $Z$ is closed in $\f{Z}$ and that the intersection of a family of closed sets of $\f{Z}$ is again closed. These observations justify the following definition: 

\begin{defn}
	Let $\f{Z}\in\z{K_L}$ and ${X\subseteq{Z}}$. The \textit{intrinsic closure} of ${X}$ in ${Z}$, denoted by ${\icl_{\f{Z}}(X)}$ is the smallest set ${X'}$ such that ${X\subseteq{X'}\subseteq{Z}}$ and ${X'}$ is closed in ${Z}$.       
\end{defn}

\subsection{Some basic properties of the rank function}

We start exploring the rank function $\delta$ in more detail. %There will be situations, for example in Section \ref{sec:ExistThm}, where we are guaranteed the structures we are interested in actually lies in $K_L$, but we wish to establish some other property about it: for example that it is in $\Kfin$. Thus it will be more convenient to explore the properties of the rank function in this broader context $K_L$.

\begin{defn}\label{defn:relRank}
	Let $\f{Z}\in \z{K_L}$ and let ${A, B}\finsubset{Z}$. Now ${\delta({B/A})=\delta{(BA)}-\delta{(A)}}$. We will call $\delta({B}/{A})$, the \textit{relative rank} of ${B}$ over ${A}$. When ${B}$ and ${A}$ are understood in context we will just say relative rank.
\end{defn}

We introduce some notation: 

\begin{notation}\label{notation:e(A)} For readability, we will often write $\z\alpha_E$ in place of ${\z\alpha}(E)$. Given $\f{Z}\in\z{K_L}$ and ${A,B,C}\finsubset Z$, we write $e({A})$ for $\sum_{E\in{L}}\z\alpha_E N_{E}({A})$, $e({A,B})$ for $\sum_{E\in{L}}\z\alpha_E N_{E}({A,B})$ and $e({A,B,C})$ for $\sum_{E\in{L}}\z\alpha_E N_{E}({A,B,C})$. 
\end{notation} 

The following collects some useful facts about the behavior of the rank function $\delta$  routine computations:

\begin{fact}\label{lem:BasicDel1}\label{lem:MonoRelRankOvBase}\label{lem:SubModularityOverBase}\label{fact:ExistenceOfStrongPoint}\label{lem:BasicDel3}\label{fact:rankAddOverBases}\label{fact:delMain}
	Let $\f{Z}\in\z{K_L}$ and let ${A,B,C,B_i}\finsubset Z$.
	\begin{enumerate}
		\item $\delta({B/A})=\delta({B})-\delta({A\cap{B}})-e({A-B,A\cap B,B-A})$ and hence if either $A$ or $B$ is in $\Kfin$, $\delta({B/A})\leq{\delta({B})-e({A-B, A\cap B, B-A})}$. Further if ${A, B}$ are disjoint then  $\delta({B/A})=\delta({B})-e({A,B})$.
		
		%\item \begin{enumerate}(!!!!!!!!!!!!!!!!!!!!!!DO NOT REMOVE!!!!!!!!!!!!!!!!!!!!!!!!!!!!1)
		
		%	\item  Assume that $B\cap C= A$. Then $\delta({BC/B})\leq{\delta({C/A})}$. Furthermore, equality holds if $B,C$ are freely joined over $A$, while $\delta({C/B})+\z\alpha_E=\delta({BC/B})+\z\alpha_E\leq{\delta({C/A})}$ whenever $E^{{BC}}\neq{E^{{B}}}\cup{E^{{C}}}$.
		
		\item Let $A'=A\cap{B}$. Now $\delta({B/A'})\geq \delta({B/A})=\delta(AB/A)$, while $\delta({AB/A})+\z\alpha_E=\delta({B/A})+\z\alpha_E\leq{\delta({B/A'})}$ whenever $E^{{AB}}\neq{E^{{A}}}\cup{E^{{B}}}$. %(set A=A', B=A and C = B in Lemma 2.3 of \cite{Las1} that has been commented out)

		\item Assume that ${BC\cap{A}}=\emptyset$, ${A\leq AB}$ and ${A\leq AC}$. Then $\delta({BC/A}) \leq \delta({B/A})+\delta({C/A})$.
		
		%\item Assume that ${A\leq B}$ and $\delta({B/A})>0$. Then there exists $b\in B-A$ such that for all $B'$ with $bA\subseteq B'$, $\delta({B'/A})>0$.
		
		\item If $\{{B}_i: i<n\}$ is disjoint over ${A}$ and ${Z}=\oplus_{i<n}{B}_i$ is their free join over $A$, then $\delta({Z/A})=\sum_{i<n}\delta({B}_i/A)$. In particular, if ${A}\leq{{B}_i}$ for each $i<n$, then ${A}\leq{\oplus_{i<n}{B}_i}$.
		
		\item $
		\begin{array}{lcl}
		\delta({B}_1{B}_2\ldots{B}_k/{A})&=&\delta({B}_1/{A})+\sum_{i=2}^k\delta({B}_i/{AB}_1\ldots{B}_{i-1})
		\end{array}
		$		  
	\end{enumerate}
\end{fact}

\section{Existence theorems}\label{sec:ExistThm}

In this section we establish several results that can be viewed as results that are purely about finite weighted hypergraphs. The results are all obtained by explicitly constructing various weighted hypergraphs. Fix an $\z{\alpha}$.  We begin with the following definitions: 

\begin{defn}\label{defn:essentialMinPa}
	Let $\f{B}\in\Kfin$ with $\delta(\f{B})>0$. We call $\f{D}\in \Kfin$ with $\f{B\subseteq D}$ an \textit{essential minimal pair} if $(\f{B,D})$ is a minimal pair and for any $\f{D'\subsetneq D}$, $\delta(\f{D'/D'\cap{B}})\geq 0$.  
\end{defn}

\begin{defn}\label{defn:RationAlpha}
	We say that $\z\alpha$ \textit{is rational} if $\z\alpha_E$ is rational for all $E\in L$.
\end{defn}

\begin{defn}\label{defn:MaxArity}
	We use $ar(L)$ to denote $\max\{ar(E):E\in L\}$.
\end{defn}

One of the main results of Section \ref{sec:ExistThm} is Theorem \ref{thm:OmitPrelim2}. It states that given $\f{B}\in\Kfin$ with $\delta(\f{B})>0$, there exists infinitely many non-isomorphic $\f{D}\in\Kfin$ where $(\f{B,D})$ is an essential minimal pair  that satisfies  $-\epsilon\leq\delta(\f{D/B})<0$ where $\epsilon$ is, \textit{in context}, arbitrarily small. The overall proof of this theorem has the following structure:\begin{enumerate}
	
	\item We begin by introducing the notion of an $L$-\textit{collection}. An $L$-collection $r$ will be a multiset, i.e. a set with repeated elements, where each element is an element of $L$. For any $E$ in $L$, we let $r(E)$ be the number of times $E$ is repeated in $r$.  
	
	\item Next we introduce the notion of a \textit{template}. A template, will be a triple $\langle n, \u{r}, t\rangle$. Here $n$ is a positive integer and $\u{r}=\langle r_1\ldots,r_n\rangle$ will index a collection $L$-collections. Further each $r_i$ will have the property that for each $E\in L$, $r_i(E)<m_{pt}$, where $m_{pt}$ is a fixed positive integer that we will introduce shortly. Finally $t=\{E_1,\ldots,E_{n-1}\}$ is an indexed $L$-collection. The idea is that the extension $\f{D}\supseteq \f{B} $ will  have universe $D-B=\{d_1,\ldots,d_n\}$. Further, for each $E\in L$, it will have $\u{r}(j)(E)$ many relations involving only subsets of $B$ and $d_j$. Also there will be precisely one relation involving $t(j)$, $\{d_j,d_{j+1}\}$ and a subset of $B$ and no other relations (besides the ones already in $\f{B}$) will hold.
	
	\item A moments' reflection shows that under the above conditions above, not all $\f{B}\in\Kfin$ will have extensions by templates (for example $L$ might contain only one relation symbol whose arity $ar(E)$ is much larger than $|B|$). We identify crude bounds such as $m_{pt}$ and on $|B|$ that will make the construction of an extension by a template feasible. Let $ar(L)=\max\{ar(E) :E\in L\}$ The bound on $|B|$ will be picked so that there are at least $m_\text{pt}ar(L)$ disjoint subsets of $B$.
	
	\item With these technical details aside, we isolate the notions of \textit{acceptable} and \textit{good} templates for a fixed $\f{B}\in\Kfin$ with positive rank. A good template $\Theta$ is set up in such a way that guarantees that an extension $\f{D}$ of $\f{B}$ using $\Theta$ will be an essential minimal pair. Thus we are left with generating good templates, which we carry out with the help of some number theoretic results (see Appendix \ref{App:NumTh}). The notion of acceptable, which is weaker than the notion of good, is isolated as it plays a part in the second main result of this section, i.e. Theorem \ref{thm:AnnhilConstruc}.	 
	
	\item We prove Lemma \ref{thm:OmitPrelim3}, which states: Given $\f{B}\in\Kfin$ with $|B|$ \textit{sufficiently large} and $\delta(\f{B})>0$ that there are here exists infinitely many non-isomorphic $\f{D}\in\Kfin$ where $(\f{B,D})$ is an essential minimal pair  that satisfies $-\epsilon\leq\delta(\f{D/B})<0$. Here again, $\epsilon$ is, in context, arbitrarily small. Finally in Theorem \ref{thm:OmitPrelim2} we establish the desired result.
\end{enumerate}

We now introduce some of the notions that we alluded to above:

\begin{defn}\label{defn:CoarseBounds}
	We define $m_\text{pt}$ be the least positive integer $m\in\omega$ such that $1-m_\text{pt}\z\alpha_E<0$ for all $E\in L$. We let $m_\text{suff}$ be the product $m_\text{pt}ar(L)$.
\end{defn}

\begin{remark}\label{rmk:ExplanationCrudeBounds}
	Note that if $\f{B}\in K_L$ and $\f{D}\in K_L$ is a one point extension of $\f{B}$ and $\delta(\f{D/B})\geq 0$, then the number of relations that include the single point in $D-B$ and $B$ is less than $m_\text{pt}$. It can be seen that given an essential minimal pair $(\f{A,C})$ and $c\in{C-A}$, then $N(c,A)<m_\text{pt}$. Now $m_\text{suff}$ gives a crude lower bound over the size of $\f{B}\in\Kfin$ over which we can construct essential minimal pairs. Here $m_\text{suff}$ stands for sufficient.       
\end{remark}

The other main result in this section, Theorem \ref{thm:AnnhilConstruc}, is concerned with building $\f{D}\in\Kfin$ such that $\delta(\f{D})=0$ that extend $\f{B}\in\Kfin$ with $\delta(\f{B})>0$. We will see that the existence of such structures can be characterized by the notion of \textit{coherence}. 

\begin{defn}\label{defn:coherence}
	We say that ${\z\alpha}$ is \textit{coherent} if there exists $\langle m_E : E\in L, m_E\in\omega, m_E>0 \rangle $ such that $\sum_{E\in L} m_{E}\z\alpha_{E} \in\mathbb{Q}$. 
\end{defn} 

\begin{remark}\label{rmk:ClarOfCoherence}
	Clearly if $\z\alpha$ is rational, then $\z\alpha$ is coherent. We now give an example of a coherent $\z\alpha$ that is not rational: Fix  $0<\beta<1/2$ irrational. If $\z\alpha(E_1)=\beta$ for some $E_1\in L$ and $\z\alpha(E_2)=1-\beta$ for some $E_2\in L$ and $\z\alpha(E) \in \{\beta,1-\beta\}$ for all $E\in L$, then $\z\alpha$ is coherent but not rational.   
\end{remark}

In Section \ref{sec:AtomECMod}, we use these structures to classify the $\z\alpha$ for which the corresponding theory of the Baldwin-Shi hypergraph has atomic models. The construction of the required $\f{D}$ will again be done with the help of templates and will reuse the ideas developed in the constructions of essential minimal pairs with some caveats. 

\subsection{Templates and Extensions}\label{subsec:Templates}

We begin by defining a template.

\begin{defn}\label{defn:collection}
	A multiset $r$ where the elements of $r$ are relation symbols from $L$ will be called an $L$-\textit{collection}. Given $E\in L$, $r(E)$ will denote the number of times that $E$ is repeated in $r$. Further we let $|r|=\sum_{E\in L} r(E)$. Given a $L$-collections $r$ and $r'$, we say that $r'$ is a \textit{sub-collection} of $L$ if $r'\subseteq r$. 
\end{defn}

\begin{notation}
	Throughout the rest of Section \ref{sec:ExistThm}, we will use the letters $r,s$ (with or without various subscripts) to denote $L$-collections.
\end{notation}

\begin{defn}\label{def:Template}
	Let $n\geq 3$ be a fixed positive integer. Let $\u{r}=\langle r_1,\ldots r_n \rangle$ where each $r_i$ is an $L$-collection. Further let ${t}$ be an \textit{indexed} $L$-collection with $|t| = n-1$, i.e. there is a fixed enumeration $E_1,\ldots,E_{n-1}$ of the elements of $t$. We call a triple $\Theta=\langle n, \u{r} , t \rangle$  an $n$-\textit{template} if for each $1\leq i \leq n$, $E\in L$ we have that $r_i(E)<m_\text{pt}$.  
\end{defn}

Given a template and $\f{B}\in{K_L}$, we use the template to create an extension $\f{D}$ of $\f{B}$. As noted previously \textit{The constructions of interest} are the ones where given $\f{B}\in\Kfin$ and we can create $\f{D}$ extending $\f{B}$ such that $\f{D}\in\Kfin$ and $\f{D}$ satisfies other desirable properties. We now make precise the notion of an extension by a template that was somewhat loosely described at the beginning of Section \ref{sec:ExistThm}.   

\begin{defn}\label{defn:ThtExt}
	Let $\f{B}\in{K_L}$ such that $|B|\geq m_\text{suff}$. Let $\Theta$ be an $n$-template. An \textit{extension of $\f{B}$ by $\Theta$} is some $\f{D}$ in $K_L$ that satisfies 
	\begin{enumerate}
		\item $\f{B\subseteq{D}}$
		\item The universe of $\f{D-B}$ is $\{d_1,\ldots,d_n\}$, i.e. it consists of $n$-points.
		\item For each $1\leq i \leq{n-1}$, there is a subset $Q\subseteq{B}$ of size $ar(E_i)-2$ such that $\{d_{i},d_{i+1}\}\cup{Q}\in{E_i^D}$ (where ${Q}$ is possibly empty).
		\item If $r_{i}(E)> 0$ for some $E\in{L}$, there are precisely $r_{i}(E)$ distinct subsets $Q_1,\ldots,Q_{r_{i}(E)}$ of $B$ of size $ar(E)-1$  such that $\{d_{i}\}\cup{Q_j}\in E^{D}$ for $1\leq j \leq r_{i}(E)$.     
		\item There are no further relations in $\f{D}$ than the ones that were originally in $\f{B}$ and the ones that are described above.
	\end{enumerate}
	In the case for any $b\in{B}$, there exists some $d_{j}$, $Q'\subseteq{D}$, $E\in L$ such that $\{b, d_j\}\cup{Q'}\in{E^\f{D}}$, we say that $\f{D}$ \textit{covers} $\f{B}$.
\end{defn}	

\begin{lemma}\label{lem:ExisOfUnQualThtExt}
	Let $\f{B}\in{K_L}$ such that $|B|\geq m_\text{suff}$. Let $\Theta$ be an $n$-template. There is an an extension $\f{D\supseteq B}$ of $\f{B}$ by $\Theta$.    
	Moreover if $\sum_{i=1}^n |r_i|\geq |B|$ or if $\sum_{ar(E)\geq 3}(t(E)+\sum_{i=1}^n r_{i}(E))\geq |B|$ there exists $\f{D}$ that covers $\f{B}$.
\end{lemma}

\begin{proof}
	Take $D_0=\{d_1,\ldots,d_n\}$ and consider the $L$ structure $\f{D}_0$ with universe $D_0$ and no relations in $\f{D}_0$. Now $\f{D}$ will be a structure with universe $B\cup{D_0}$.  
	
	First note that since $|B|\geq m_\text{suff}$, $B$ has at least $m_\text{pt}$ distinct subsets of size $ar(E)-1$ for each $E\in{L}$. For each $1\leq i \leq n-1$ we may fix some subset $Q\subseteq B$ and add a relation so that $\{d_{i},d_{i+1}\}\cup{Q}\in{E_i^D}$. Here $Q$ is possibly empty: in fact $Q$ is empty if and only if $E_i$ is a binary relation symbol.  
	
	Now fix $1\leq i \leq{n}$. For each $E\in{L}$ we have $r_{i}(E)<m_\text{pt}$. Thus for fixed $E\in{L}$, as $|B|\geq m_\text{suff}$, we may choose $r_{i}(E)$ distinct subsets $Q_j$ as $1\leq j\leq r_i(E)$, of $B$ where each $Q_j$ is of size $ar(E)-1$. Add relations so that $\{d_{i}\}\cup{Q_j}\in E'^{D}$ for $1\leq j \leq r_{i}(E)$. Do this for each relation symbol $E\in{L}$. Now assume that this process of adding relations has been carried out for each $1\leq i \leq n$. Let the resulting structure be $\f{D}$.  Note that the relations that hold on $\f{D}$ are precisely the ones that turn $B$ to $\f{B}$ and the relations described so far. It is now clear that the resulting structure satisfies the properties required of  $\f{D}$.
	
	If $\sum_{i=1}^{n} |r_i|\geq |B|$ we may insist that the choice of $Q_j$, as $E$ ranges through $L$,  be made so that their union is $B$.  If $\sum_{ar(E)\geq 3}(t(E)+\sum_{i=1}^n r_{i}(E))\geq |B|$, then we may insist that the choice of the various $Q$ and $Q_j$ be made so that the union is $B$. In either case the statement that for any $b\in{B}$, there exists some $d_{j}$, $Q'\subseteq{D}$ such that $\{b, d_j\}\cup{Q'}\in{E^\f{D}}$ for some $E\in{L}$ holds.      
\end{proof}

\begin{remark}
	Note that an extension by $\Theta$ need not be unique up to isomorphism over $\f{B}$. However given two non-isomorphic extensions $\f{D,D'}$ of $\f{B}$ by $\Theta$ their relative ranks are identical: $\delta(\f{D/B})=\delta(\f{D'/B})$. Hence $\delta(\f{D})=\delta(\f{D'})$.
\end{remark}

\begin{notation}\label{notation:ThetaExtension}
	Let $\Theta = \langle n,\u{r},t \rangle$ be an $n$-template. Fix $1\leq j\leq n$. Let $\f{B}\in{K_L}$ such that $|B|\geq m_\text{suff}$  and let $\f{D}$ be an extension by $\Theta$ of $\f{B}$. Under the natural enumeration of $D-B=\{d_1,\ldots,d_n\}$ used to construct the extension; we let $\f{D}^j$ denote the substructure of $\f{D}$ with universe $B\cup\{d_{1},\ldots,d_{j}\}$ for $1\leq j\leq n$ and we let ${\f{D}^{j,k}}$ denote the substructure of $\f{D}$ with $=B\cup\{d_j,\ldots d_{k}\}$ for any $1\leq j\leq k\leq n$.    	        
\end{notation}

We now define the \textit{acceptable} and \textit{good} templates. As noted previously, good templates are defined with the construction of essential minimal pairs in mind. Acceptable templates capture a weaker notion that is common to both the essential minimal pairs and the rank zero extensions that are dealt with in Section \ref{subsec:Coherence}. 

When dealing with templates it will often be convenient to focus on the sub-language of the symbols that occur in $\Theta$. We make the following somewhat broader definition.     

\begin{defn}\label{defn:localization}
	Given a \textit{triple} $\Theta=\langle n,\u{r},t \rangle$, the \textit{localization of $L$ to $\Theta$}, denoted by $L^\Theta$ is the subset of $L$ such that $E\in{L^\Theta}$ if and only if $E$ occurs positively in $\Theta$, i.e.  $r_j(E)>0$ for some $1\leq j \leq n$ or $E=E_j$ for some $1\leq j \leq n-1$. Further we let $Gr_{\Theta}(2)$ denote the least positive value of $\sum_{E\in L^\Theta}\z\alpha(E)n_{E}-1$ for non-negative integers $n_E$.
\end{defn}

\begin{remark}
	The reason behind using the notation $Gr_{\Theta}(2)$ will become clear in Section \ref{subsec:GeneratingTemplates}. 
\end{remark}

\begin{defn}\label{defn:GoodTemplate}
	Let $\f{B}\in \Kfin$ be such that $|B|\geq m_\text{suff}$ and $\delta(\f{B})>0$.  Let $\Theta$ be a $n$-template and let $\f{D}$ be an extension of $\f{B}$ by $\Theta$. We say that $\Theta$ is \textit{acceptable} for $\f{B}$ if and only if 
	\begin{enumerate}
		\item $0<-\delta(\f{D/B})\leq \min\{\delta(\f{B}), Gr_{\Theta}(2)\}$. 
		\item $\delta(\f{D}^j/\f{B})\geq 0$ for $1\leq j \leq n-1$.
		\item $\z\alpha(E_j)-\delta(\f{D}^j/\f{B})>0$ for $1 \leq j \leq n-1$.
	\end{enumerate}
	We say that $\Theta$ is \textit{good} for $\f{B}$ if \begin{enumerate}
		\item $\Theta$ is acceptable for $\f{B}$.
		\item $\z\alpha(E_j)-\delta(\f{D}^j/\f{B})+\delta(\f{D/B})\geq 0$ for $1\leq j  \leq n-1$.
		\item We may in addition assume that $\f{D}$ can be chosen so that it covers $\f{B}$. 
	\end{enumerate}   
\end{defn}

The following lemma captures the key properties of extensions by acceptable and good templates.

\begin{lemma}\label{lem:preOmitPrelim}
	Let $\f{B}\in \Kfin$ be such that $|B|\geq m_\text{suff}$ and $\delta(\f{B})>0$. Let $\Theta$ be an $n$-template and let $w=n-(\sum_{i=1}^{n-1}\z\alpha_{E_i}+\sum_{i=1}^{n}\sum_{E\in L}\z\alpha_E{r}_{i}(E))$. Let $\f{D}$ be an extension by $\Theta$ of $\f{B}$
	\begin{enumerate}
		\item If $\Theta$ is acceptable, then 
		\begin{enumerate}
			\item[1.a] For any $\f{B\subseteq D'\subsetneq D}$ such that $d_n\notin{\f{D'}}$, $\delta(\f{D}'/\f{B})\geq 0$
			\item[1.b] For any $\f{D'\subsetneq D}$ such that $d_n\notin{\f{D'}}$, $\delta(\f{D}'/\f{D'}\cap{\f{B}})\geq 0$
			\item[1.c] For any $\f{B\subseteq D'\subseteq D}$, $\delta(\f{D}'/\f{B})\geq w$
		\end{enumerate}
		\item If $\Theta$ is good for $\f{B}$, we may choose $\f{D}$ so that $\f{D}$ covers $\f{B}$ and then
		\begin{enumerate}
			\item[2.a] $\f{D}\in\Kfin$
			\item[2.b] For any proper $\f{B\subseteq D'\subsetneq D}$, $\delta(\f{D}'/\f{B})\geq 0$
			\item[2.c] For any $\f{D'\subsetneq D}$, $\delta(\f{D'/B\cap{D'}})\geq 0$
		\end{enumerate}
		i.e. $(\f{B,D})$ is an essential minimal pair with $\delta(\f{D/B})=w$.
	\end{enumerate}   
\end{lemma}

\begin{proof}
	
	We begin with $(1)$: For $(1.a)$, the case $\f{D'}=\f{D}^{j}$ for some $1\leq j \leq n-1$ is immediate.  Consider the case that $\f{D'}=\f{D}^{k+1,j}$ where $1\leq k < j \leq n-1$. As there is only a single relation, namely $E_k$,  that contains the points $d_k,d_{k+1}$, it follows that $\delta(\f{D}^{k+1,j}/\f{B})=\delta(\f{D}^{j}/\f{D}^{k})+\z\alpha(E_k)$. Further    $\delta(\f{D}^{k+1,j}/\f{B})=\delta(\f{D}^{j}/\f{B})-\delta(\f{D}^{k}/\f{B})+\z\alpha(E_k)$. But $\z\alpha(E_k)-\delta(\f{D}^{k}/\f{B})+\delta(\f{D}^{j}/\f{B})\geq 0$ by using conditions $2$ and $3$ of $\Theta$ being acceptable. Since an arbitrary $\f{B}\subseteq\f{D'}\subsetneq\f{D}$ with $d_n\notin\f{D}$ can be written as the free join different $\f{D}^{k,j}$ over $\f{B}$, it follows that \textit{for} $\f{B\subseteq D'\subsetneq D'}$, $\delta(\f{D'/B})\geq 0$. Now consider an arbitrary $\f{D'\subseteq D}$ such that $d_n\notin\f{D'}$ and $ \f{B\not\subseteq D'}$. By $(2)$ of Fact \ref{lem:MonoRelRankOvBase} $\delta(\f{D'/B\cap{D'}})\geq \delta({BD'/B})$. But the above shows that $\delta({BD'/B})\geq 0$ and thus $(1.b)$ follows.
	
	For $(1.c)$, note that for $1\leq j\leq n$, $\delta(\f{D}^j/\f{B})<0$ if and only if $j=n$ if and only if $\delta(\f{D}^j/\f{B})=w$. As $\delta(\f{D}^{k+1,j}/\f{B})=\delta(\f{D}^j/\f{B})+\z\alpha(E_k)-\delta(\f{D}^k/\f{B})$ for $1\leq k <j\leq n$ and since $\f{D'}$ can be written as the free join of several $\f{D}^{k,j}$ and over $\f{B}$ and at most one of the $\f{D}^{k,j}$ satisfies $0>\delta(\f{D}^{k,j}/\f{B})\geq w$, it follows that $\delta(\f{D'/B})\geq w$. \\
	
	\noindent Now consider $(2)$: We are assuming $\f{D}$ covers $\f{B}$. As $\delta(\f{{D/B}})=w$ by construction both $(2.a)$ and the statement regarding $(\f{B,D})$ being an essential minimal pair follows from $(2.b)$ and $(2.c)$. For the proof of $2.b$, first consider $\f{D'}=\f{D}^{j+1,n}$ for $1\leq j \leq n-1$. By arguing as above we obtain that $\delta(\f{D}^{j+1,n}/\f{B})=\z\alpha(E_j)-\delta(\f{D}^j/\f{B})+\delta(\f{D/B})$. By using condition $(2)$ of good, it follows that $\z\alpha(E_j)-\delta(\f{D}^j/\f{B})+\delta(\f{D/B})\geq 0$. As $\Theta$ is good, it is also acceptable and thus $\delta(\f{D}^{k,j}/\f{B})\geq 0$ for $1\leq k \leq j \leq n-1$. Since an arbitrary $\f{B}\subseteq\f{D'}\subsetneq\f{D}$ can be written as the free join different $\f{D}^{k,j}$ over $\f{B}$ it follows that \textit{for} $\f{B\subseteq D'\subsetneq D'}$, $\delta(\f{D'/B})\geq 0$. $\delta(\f{D'/B})\geq 0$.    
	
	It remains to show that for a general substructure $\f{D'\subsetneq{D}}$, we have that $\delta(\f{D'/B\cap{D'}})\geq{0}$. If $D'-B\neq D-B$, then this follows easily by $(1.b)$ and $(2)$ of Fact \ref{lem:BasicDel1}. So assume that $D'-B = D-B$.  Since $\f{D'\subsetneq D}$, it follows that $\f{D'\cap{B}}\neq \f{B}$. Fix a relation $E\in L$ such that it holds with a point from $D'-B$ and at least one point from $B-B'$. By using $(2)$ of Fact \ref{lem:BasicDel3} we see that $\delta(\f{D'}/\f{D'\cap B})\geq \delta(\f{D/B})+\z\alpha(E)$. Since $-Gr_{\Theta}(2)\leq \delta(\f{D/B})$, it follows that $0 \leq Gr_{\Theta}(2)+\z\alpha_E \leq \delta(\f{D/B})+\z\alpha_E$. Thus $(2.c)$ follows. %By Lemma \ref{lem:RemovalOfPoint},  %Now $\delta(\f{D'}/\f{D'\cap B})=\delta({D'-B}/{D'\cap B})=\delta(D'-B)-e(D'-B, D'\cap B)$ using $(1)$ of Fact \ref{lem:BasicDel1}.
\end{proof}

\subsection{Generating Templates}\label{subsec:GeneratingTemplates}

In this section we introduce the notions of \textit{acceptable pairs} and \textit{good pairs}. We will show how to construct a good/acceptable template by using a good/acceptable pair. The acceptable and good pairs are easily obtained by the well known number theoretic results that can be found in the Appendix. This allows us to establish that the constructions in Section \ref{subsec:Templates} can indeed be carried out. We finish this section with Lemma \ref{lem:OmitPrelim} and Theorem \ref{thm:OmitPrelim2} which generalize results in \cite{Las1}. We begin by introducing the notion of granularity. %The results here are central to the quantifier elimination results in Section \ref{sec:QuantElim}.

\begin{defn}\label{defn:granularity}
	Given $m\in{\omega}$ with $m\geq 2$ and $L_0\subseteq L$, we define $Gr_{L_0}(m)$, \textit{the granularity $m$ relative to $L_0$}, to be the smallest positive value $\sum_{E\in{L_0}}\z\alpha_E n_E -k$ where $k$ is an integer satisfying $0<k<m$ and each $n_E\in\omega$. In case $L=L_0$ we call $Gr_{L}(m)$ the \textit{granularity of} $m$ and denote it by $Gr(m)$.
\end{defn}

\begin{remark}\label{rmk:GranulairtyAndDropInRank}
	Let $m\in\omega$ with $m\geq 2$ and $\f{A,B}\in\Kfin$. If $|B-A|<m$, then $\delta(\f{B/A})\leq -Gr(m)$. This observation is crucial for many of the arguments in \cite{Las1}.    
\end{remark}

\begin{remark}
	Note that given a triple $\Theta=\langle n,\u{r}, t\rangle$, $Gr_{\Theta}(2)=Gr_{L^\Theta}(2)$. Further if $Gr(2)=\sum_{E\in L}n_E\z\alpha_E-1$, then $\sum_{E\in L}n_E < m_{pt}$ 
\end{remark}

The following is immediate from the definition of granularity.

\begin{lemma}\label{lem:Gr1AtMostAlpha}
	For all $E\in{L}$, $Gr(2)\leq \z\alpha_E$.
\end{lemma}

We now turn our attention to good pairs and acceptable pairs. The goal will be to use good/acceptable pairs to generate good/acceptable templates, which we proceed to do in Lemma \ref{lem:preOmitPrelim2}.

\begin{defn}\label{defn:WeightedSum}
	Given a non-negative integer $n$ and an $L$-collection $r$, we let the weighted sum $n-\sum_{E\in L}\z\alpha_{E}r(E)$ be denoted by $w(n,r)$.	
\end{defn}

\begin{defn}\label{defn:GoodPairs}
	Let $\f{B}\in\Kfin$ with $\delta(\f{B})>0$. Let $n\in\omega$ and let $s$ be an $L$-collection. Let $L_0\subseteq L$ be such that $E\in L_0$ if and only if $s(E)>0$.  We say that $\langle n, s \rangle$ is an \textit{acceptable pair} for $\f{B}$, if \begin{enumerate}
		\item $\min\{\delta(\f{B}), Gr_{L_0}(2) \}\geq -w(n,s)> 0$ 
		\item $|s|\geq n$ 
	\end{enumerate}
	We say that $\langle n,s \rangle$ is a \textit{good pair} for $\f{B}$
	\begin{enumerate}
		\item $\langle n,s \rangle$ is acceptable
		\item $|s|\geq{|B|+(n-1)}$ or $\sum_{ar(E)\geq 3}(t(E)+\sum_{i=1}^n r_{i}(E))\geq |B|$ 
		\item For all $m \leq n$ and sub-collections $s'$ of $s$, $w(m, s')$ not in the interval $(w(n,s),0)$. %$w(m,s')\notin (w(n,s),0)$. %
	\end{enumerate}
	Often we will not mention $\f{B}$ as it will be clear from context.
\end{defn}

\begin{lemma}\label{lem:preOmitPrelim2}
	Let $\f{B}\in{\Kfin}$ with $\delta(\f{B})>0$, $|B|\geq m_\text{suff}$. If $\langle n,s \rangle$ is an acceptable pair for $\f{B}$, then there exists an acceptable $n$-template $\Theta =\langle n,\u{r},t\rangle$. If $\langle n,s \rangle$ is good, then $\Theta$ will be good for $\f{B}$. 
\end{lemma}

\begin{proof}
	We begin with the observation that if $u$ is a sub-collection of $s$, then $s-u$ is the residual multiset with $(s-u)(E)=s(E)-u(E)$. Our first goal is to define the \textit{triple} $\Theta=\langle n,\u{r}, t\rangle$. We do this in Step 1. We do this using a ``greedy algorithm". In Step 2, we establish that the triple $\Theta$ we have constructed is indeed a template and it is acceptable/good based on the corresponding properties of $(n,s)$.\\   
	
	\noindent\textit{Step 1}: We first define $t$. For $1\leq j \leq n-1$ inductively define $E_j$ so that $E_j$ is in the residual multiset $s-\{E_1,\ldots,E_{j-1}\}$ and $\alpha(E_j)=\max\{\alpha(E):E\in s-\{E_1,\ldots,E_{j-1}\} \}$. If there is $E\in{L}$ with arity at least 3 such that  $s(E)\geq n-1 \geq |B|$ and $\z\alpha(E)\geq \z\alpha(E^*)$ for all $E^*\in L$, then we insist that the above $E_j$ satisfy $E_j=E$. Let $t$ be the \textit{ordered} $L$-collection $\langle E_1\,\ldots,E_{n-1} \rangle$. Let $s_1$ be the residual multiset $s-\{E_1,\ldots,E_{n-1}\}$.  For $1\leq j\leq n$ define the \textit{potential relative rank} $Rel(j) = \sum_{i=1}^j w(1,r_i)-\sum_{i=1}^{j-1}\z\alpha(E_i)$. 
	
	First let $r_1\subseteq s_1$ be an $L$-collection such that $Rel(1)=w(1,r_1)$ achieves the least possible non-negative value. Assume that for $1\leq j \leq n-1$ that $r_j, s_j $ have been defined and take $s_{j+1}$ to be the residual multiset $s_{j}-r_{j}$. For $1\leq j< n-1$ pick $r_{j+1}\subseteq s_{j+1}$ such that  $Rel(j+1) = Rel(j) + w(1,r_{j+1})-\alpha(E_{j})$ attains the least possible non-negative value and let $r_n=s_n$. Let $\u{r}=\langle r_1,\ldots, r_{n}\rangle$ and let $\Theta$ be the triple $\langle n,\u{r}, t \rangle$. \\
	
	\noindent\textit{Step 2}: We first show that $\Theta$ is indeed an $n$-template. We begin with the following claims.\\
	
	\noindent\textit{Claim 1}: \textit{For $1\leq j < n$, $s_{j+1}$ is non-empty}: We begin by noting that as $|s|\geq n$, $s_1$ is non-empty. Now assume to the contrary that $s_{j+1}$ is empty for some $1\leq j < n$ and let $j_0$ be the least positive integer for which $s_{j_0+1}$ is empty. Then for all $j'\geq j_0+1$, $s_j'$, $w(1,r_{j'})=1$. Now it follows that $0>w(n,s)=Rel(n)=Rel(j_0) + (n-j_0) -\sum_{i=j_0}^{n-1}\z\alpha(E_i)$. By construction $Rel(j_0)\geq 0$. Further as for each $E\in L$, $\z\alpha{(E)}\leq 1$ implies that $(n-j_0) -\sum_{i=j_0}^{n-1}\z\alpha(E_i)\geq 0$. But this yields a contradiction that proves the claim. \\ 
	
	\noindent\textit{Claim 2}: \textit{For $1\leq j < n$, $Rel(j)<\z\alpha(E_j)$}: If not, $Rel(j)\geq \z\alpha(E_j)$ for some $1\leq j <n$. From Claim 1 it follows that there is some $E\in L^\Theta$ such that $s_{j+1}(E)>0$. By our choice of the $E_i$, it follows that $\z\alpha(E_j)\geq \z\alpha(E)$. However this shows that $Rel(j)-\z\alpha(E)\geq \z\alpha(E_j)-\z\alpha(E)\geq 0$ which contradicts our choice of $r_j$.\\ 
	
	Note that to show that $\Theta$ is an $n$-template it suffices to show that for $1\leq j\leq n$, $w(1,r_j)\geq 0$. Now for all $1\leq j <n-1$, $Rel(j+1)\geq 0$ and $Rel(j)<\z\alpha(E_j)$ yields that  $w(1,r_{j+1})=Rel(j+1)+\z\alpha(E_j)-Rel(j)\geq 0$. Now assume that $w(1,r_n)<0$. Then $w(1,r_{n})\leq -Gr_{\Theta}(2)$. Now $Rel(n)=w(1,r_n)+Rel(n-1)-\z\alpha(E_{n-1})<-Gr_{\Theta}(2)$ which contradicts $-Rel(n)\geq Gr_{\Theta}(2)$. Thus it follows that $w(1,r_n)\geq 0$. Hence $\Theta$ is indeed a $n$-template. 
	
	Let $\f{D}$ be an extension of $\f{B}$ by $\Theta$ as given by Lemma \ref{lem:ExisOfUnQualThtExt}. Observe that $\delta(\f{D}^j/\f{B})=Rel(j)$ for $1\leq j \leq n$. It immediately follows that if $\langle n,s \rangle$ is acceptable, then $\Theta$ is also acceptable. Now assume that $\langle n,s \rangle$ is good. We claim that $\Theta$ is good. By construction $|s|=|t| + \sum_{i=1}^n |r_i|$. Recall condition $(2)$ of good. If $|s|\geq |B|+(n-1)$, then $\sum_{i=1}^n |r_i|\geq |B|$. Else we have that $\sum_{ar(E)\geq 3}(t(E)+\sum_{i=1}^n r_{i}(E))\geq |B|$. Now  Lemma \ref{lem:ExisOfUnQualThtExt} shows that $\f{D}$ can be constructed in a manner covers $\f{B}$. Thus in order to establish that $\Theta$ is good it suffices to show $\z\alpha(E_j)-\delta(\f{D}^j/\f{B})+\delta(\f{D/B})\geq 0$ for $1\leq j  \leq n-1$. Suppose to the contrary that $a=\z\alpha(E_j)-\delta(\f{D}^j/\f{B})+\delta(\f{D/B}) < 0$ for some $1\leq j  \leq n-1$. Thus we may write $a=w(m,s')$ for some $m \leq n$ and some sub-collection $s'$ of $s$. Now by clause $(3)$ of goodness and the fact that $\langle n,s \rangle$ is good, it follows that $a\leq w(n,s)$. But $w(n,s)=\delta(\f{D/B})$ and hence $\z\alpha(E_j)-\delta(\f{D}^j/\f{B})\leq 0$, a contradiction to Claim 2. Thus $\Theta$ is good. \end{proof}

\begin{cor}\label{lem:OmitPrelim}
	Let $\f{B}\in{\Kfin}$ with $\delta(\f{B})>0$, $|B|\geq m_\text{suff}$ and $\langle n,s \rangle$ a good pair with $n\geq 3$. Then there is an $\f{D}\in{\Kfin}$ such that $(\f{B,D})$ is an essential minimal pair with $w(n,s)=\delta(\f{D/B})<0$.
\end{cor}

\begin{proof}
	This follows directly from Lemma \ref{lem:preOmitPrelim} and \ref{lem:preOmitPrelim2}.
\end{proof}

In Remark \ref{rmk:GranulairtyAndDropInRank}, we established a link between the relative rank of structures and granularity. As it turns out, granularity offers us a very convenient way of establishing a connection between acceptable/good pairs and the number theoretic facts in the Appendix (See Lemma \ref{lem:GoodPairsAndGranularity} and Theorem \ref{thm:OmitPrelim2} below). Thus granularity takes on two \textit{separate} roles: it's original role in \cite{Las1} and the one just mentioned (replacing the role played by \textit{local optimality}, in Section 4 of \cite{Las1}). There is no interaction between the different roles.

We now turn our attention towards using the number theoretic results in the Appendix to construct good pairs.

\begin{lemma}\label{lem:granularity}
	The sequence given by $\langle Gr(m) :m\in{\omega}\rangle$ is a monotonic decreasing sequence. If $\z\alpha$ is not rational, then  $\langle Gr(m) :m\in{\omega}\rangle$ converges to $0$. If $\z\alpha$ is rational, then $Gr(m)$ is eventually constant with $Gr(m)=1/c$ for sufficiently large $m$.
\end{lemma}

\begin{proof}
	If $\z\alpha$ is not rational then there is some $E\in L$ such that $\z\alpha_E$ is irrational. Now the required result follows from Remark \ref{rmk:ContinuedFractions}. If $\z\alpha$ is rational, then the required result follows from Remark \ref{rmk:InfSolDioEqn}. 
\end{proof}

\begin{notation}
	We fix some notation: Whenever the assumption that $\z\alpha$ is rational is in effect, we assume that $\z\alpha_E=\frac{p_E}{q_E}$ in reduced form and that $c=\text{lcm}(q_E)$.
\end{notation}

\begin{lemma}\label{lem:GoodPairsAndGranularity}
	Let $n\in\omega$ with $n\geq 3$ and $s$ be an $L$-collection. For $1\leq m \leq n$ and any sub-collection $s'$ of $s$, $w(m,s')$ is not in the interval $(-Gr(n+1),0)$.  
\end{lemma}

\begin{proof}
	Let $n,s,m,s'$ be as above. As granularity is monotonically decreasing, $Gr(n+1)\geq Gr(m+1)$. Assume to the contrary that $w(m,s')\in(-Gr(n+1),0)$. This yields that $Gr(n+1)>w(m,s')>0$. But $w(m,s')\geq Gr(m+1)>0$, a contradiction which established the claim. 
\end{proof}

\begin{lemma}\label{lem:NumberTheoreticOmitPrelim2}
	Let $\f{B}\in{\Kfin}$ with $\delta(\f{B})>0$ and $|B|\geq m_\text{suff}$. 
	\begin{enumerate}
		\item Let $\epsilon>0$ and assume that $\z\alpha$ is not rational. Then for any $E\in L$ such that $\z\alpha_{E}$ is irrational, there are infinitely many good pairs $(n,s)$ for $\f{B}$ such that $0<-w(n,s)<\epsilon$ and $s$ is such that $s(E^*)>0$ if and only if $E^*=E$ for all $E\in L$.  
		\item If $\z\alpha$ is rational, then we may obtain infinitely many good pairs $(n,s)$ for $\f{B}$ such that $-w(n,s)=1/c$. 
	\end{enumerate}
\end{lemma}

\begin{proof}
	\noindent(1): Let $E\in L$ be such that $\z\alpha(E)$ is irrational. Let $L'=\{ E \}$ and let $\alpha=\z{\alpha}(E)$. Note that we may as well assume that $\epsilon\leq \min\{\delta(\f{B}),Gr_{L'}(2)\}$. As $\lim_{n}Gr_{L'}(n)=0$, there is an infinite set $A$ of positive integers such that $Gr_{L'}(n+1)<Gr_{L'}(k)$ for all $2\leq k\leq n$. For each $n\in A$, let $l_n$ be such that $Gr_{L'}(n+1)=l_n\alpha-n$. Since $\epsilon$, $|B|$ are fixed \textit{and} $\alpha<1$, all but finitely many $n\in A$ satisfy $0<l_n\alpha-n<\epsilon$ and $l_n\geq |B|+(n-1)$. Given such $n$, let $s$ be the $L$-collection that contains $l_n$ many $E$ relation symbols and no other relation symbols. It is immediate that by our choice of $n$ and $s$ that $(n,s)$ is a good pair with $0<-w(n,s)<\epsilon$ and that $s$ satisfies the other properties given in $(1)$.\\ 
	
	\noindent\textit{(2)} : Assume that $\z\alpha$ is rational. The proof now splits off into two cases depending on the value of $c$.
	
	First consider the case $c>1$: Then $Gr(n')=1/c<1$ for all sufficiently large $n'$. Note that $\delta(\f{B})=k/c$ for some $k\in\omega, k\neq 0$ and thus $\delta(\f{B})\geq 1/c$. Let $L'=\{L\in E : \z\alpha_E < 1\}$. Using Remark \ref{rmk:InfSolDioEqn} of the Appendix, there is an infinite set $A$ of positive integers $n$ such that $Gr_{L'}(n+1)=1/c$. For each $n\in A$, let $l_n:L'\rightarrow\omega$ be a function such that $Gr_{L'}(n+1)=\sum_{E\in L'}l_n(E)\alpha_{E} - n$. Since $|B|$ is fixed \textit{and} $\z\alpha_E<1$ for each $E\in L'$, all but finitely many $n\in A$ satisfy $\sum_{E\in L'}\z\alpha_E l_n(E) -n=1/c$ and $\sum_{E\in L'}l_n(E)\geq |B|+(n-1)$. Given such $n$, let $s$ be the $L$-collection that contains exactly $l_n(E)$ many $E$ relation symbols for $E\in L'$ and no other relation symbols. Now by our choice of $n,s$ it is immediate that $(n,s)$ is a good pair with $-w(n,s)=1/c$. 
	
	Now consider the case $c=1$: Now for each $E\in{L}$, $\z\alpha(E)=1$, $Gr(m)=1$ for all $m\geq 2$ and all finite structures have integer rank. Note that there is some $E\in L$ that has arity at least $3$ as $\z\alpha(E)=1$ for each $E\in L$ implies that arity of each relation symbol cannot be $2$. Fix such an $E\in L$ and let $L'=\{E\}$. Then for any $n\geq |B|+1$ take $s$ to be the $L$-collection with $n$ many $E$ relations and no other relations. A routine verification shows that $\langle n,s \rangle$ is a good pair. 
\end{proof}

We now put the previous results together to establish:  

\begin{lemma}\label{thm:OmitPrelim3}
	Let $\f{B}\in{\Kfin}$ with $\delta(\f{B})>0$ and $|B|\geq m_\text{suff}$. 
	\begin{enumerate}
		\item Let $\epsilon>0$ and assume that $\z\alpha$ is not rational. Now given any $E\in L$ such that $\z\alpha_{E}$ is irrational, we can construct infinitely many non-isomorphic $\f{D}\in\Kfin$ such that $(\f{B,D})$ is an essential minimal pair that satisfies $-\min\{\epsilon,\delta(\f{B})\}<\delta(\f{D/B})<0$ where the new relations that appear in $\f{D}$ that were not in $\f{B}$ are $E$ relations.
		\item If $\z\alpha$ is rational, then we can construct infinitely many non-isomorphic $\f{D}\in\Kfin$ such that $(\f{B,D})$ is an essential minimal pair that satisfies $\delta(\f{D/B})=-1/c$.
	\end{enumerate}   
\end{lemma}

\begin{proof}
	Use Lemma \ref{lem:NumberTheoreticOmitPrelim2} to obtain a good pair $(n,s)$ for $\f{B}$ that satisfies $0<-w(n,s)\leq Gr(m)$. Now use Corollary \ref{lem:OmitPrelim} to construct an essential minimal pair $(\f{B,D})$ with $w(n,s) = \delta(\f{D/B})<0$. As $(n,s)$ is a good pair, $\f{D}\in\Kfin$.  We can obtain infinitely many $\f{D}$ as required by varying our choice of good pairs. Further $(1), (2)$ can be obtained by choosing suitable good pairs using $(1), (2)$ (respectively) of Lemma \ref{lem:NumberTheoreticOmitPrelim2}.      
\end{proof}

The two clauses of the following lemma illustrate some routine argument patterns that can be used in constructing new structures by taking free joins. It will also yield a substantial part of Theorem \ref{thm:OmitPrelim2} and Lemma \ref{lem:Omit1}. 

\begin{lemma}\label{lem:NewStrsViaFreeJoins}
	Let $\f{A,B}\in\Kfin$ with $\f{A\leq B}$. Assume that $(\f{B,C})$ is an essential minimal pair and let $\gamma=-\delta(\f{C/B})$. Then \begin{enumerate}
		\item We can construct $\f{D}\in\Kfin$ such that $\f{B\subseteq D}$, $\f{A\leq D}$ and $0\leq \delta(D/A)<\gamma$. Further if $(\f{B,G})$ is a minimal pair with $|G|<|C|$, then $\f{G}$ does not embed into $\f{D}$ over $\f{B}$.  
		\item Assume that $\delta(\f{A})\geq \gamma$.  Then we can construct $\f{D}\in\Kfin$ such that $\f{B\subseteq D}$, $(\f{A,D})$ is an essential minimal pair that satisfies $0>\delta(\f{D/A})\geq -\gamma$
	\end{enumerate} 	
\end{lemma}   

\begin{proof}
	\noindent  Note that there is some \textit{non-negative} integer $k$ such that $k\gamma\leq \delta(\f{B/A}) < (k+1)\gamma$. Let $\f{D}$ be the free join of $k$-copies of $\f{C}$ over $\f{B}$ and enumerate the copies of $\f{C}$ in $\f{B}$ by $\{\f{C}_i:1\leq i\leq k \}$ (with $\f{B=D}$ if $k=0$). We now show that $\f{D}$ has the required properties. We begin by establishing some notation: Let $\f{D'\subseteq{D}}$ be a nonempty substructure of $\f{D}$ and let $\f{C}_{i}'=\f{C}_{i}\cap{\f{D}'}$ and $\f{B'}=\f{D'\cap{B}}$.
	
	Clearly $\f{B\subseteq D}$ and $\f{D}\in K_L$. By Remark \ref{rmk:ConstrVerifSimpl}, $\f{D}\in\Kfin$ follows if you show that $\f{A\leq D}$. This is equivalent to establishing $\delta(\f{D'/A})\geq 0$ in the case that $\f{A\subseteq D'}$. So we will assume that $\f{A\subseteq D'}$. Since $\f{A\leq B}$, if $\f{ D'\subseteq B}$, we have the required result. So consider $\f{D'}\nsubseteq \f{B}$. We may view $\f{D'}$ as the free join of $\f{D}_{i}'$ over $\f{B}'$. Note that  $\delta(\f{D'/B'})=\sum_{i=1}^k\delta(\f{C}_i'/\f{B}')$ by  $(4)$ of Fact \ref{lem:BasicDel3}. Since $(\f{B,C})$ are essential minimal pairs, it follows that if $\f{B'\neq{B}}$, then $\delta(\f{C}_{i}'/\f{B'})\geq{0}$. Further if $\f{B'=B}$, then $\delta(\f{D'/B})\geq{-k\gamma}$ with equality holding if and only if $\f{D'=D}$. 
	
	Assume that $\f{A\subseteq{D'}\subseteq{D}}$. We need to establish that  $\delta(\f{D'/A})\geq{0}$. First consider the case where $\f{B'\neq{B}}$. Now $\delta(\f{D}'/\f{B'})\geq{0}$. Further $\delta(\f{D'/A})=\delta(\f{D'/B'})+\delta(\f{B'/A})$. Since $\f{A\leq{B}}$ and $\f{A\subseteq{B'}\subseteq{B}}$, we have that $\delta(\f{B'/A})\geq{0}$. Thus $\delta(\f{D'/A})\geq{0}$. Now consider the case  $\f{B'=B}$. In this case we have that  $\delta(\f{D'/A})=\delta(\f{D'/B})+\delta(\f{B/A})\geq{-k\gamma+\delta(\f{B/A})}\geq{0}$. Hence $\f{A\leq D}$.
	
	A simple calculation yields  $\delta(\f{D/A})=-k\gamma+\delta(\f{B/A})<\gamma$. We now show that \textit{no} $\f{G}$ such that $(\f{B,G})$ is a minimal pair with $|G|<|C|$ embeds into $\f{D}$ over $\f{B}$. Assume such a minimal pair did embed into $\f{D}$ over $\f{B}$ and let its image be $\f{D}'$. Now $\delta(\f{D'/B})=\sum_{i=1}^k\delta(\f{C}_i'/\f{B})$. But \textit{each} $\delta(\f{C}_i'/\f{B})\geq 0$ unless $\f{C}_i'=\f{C}$. Thus $|D'|\geq |C|$, a contradiction.\\
	
	\noindent (2) Note that there is some \textit{non-negative} integer $k$ such that $k\gamma\leq \delta(\f{B/A})<(k+1)\gamma$. Consider the structure $\f{D}$ which is the free join of $k+1$-copies of $\f{C}$ over $\f{A}$ where . Enumerate these copies of $\f{C}$ as $\f{C}_{1}\ldots \f{C}_{k+1}$. Let $\f{D'\subseteq{D}}$ be non-empty, $\f{B'}=\f{B\cap{D'}}$ and $\f{C}_i'=\f{C}\cap{D}'$ 
	
	We begin by showing that $\f{D}\in\Kfin$. We need to show that $\delta(\f{D'})\geq 0$. As this is immediate when $\f{D'\subseteq B}$, we may as well assume that this is not the case. Now as in $(1)$,  $\delta(\f{D'/B'})=\sum_{i=1}^{k+1}\delta(\f{C}_i'/\f{B'})$. As $(\f{B,C})$ is an essential minimal pair we need only consider $\f{B'=B}$ (the other case follows easily as in $(1)$). Then $\delta(\f{D'/B})\geq -(k+1)\gamma$. But by our choice of $k$ and using the assumption $\delta(\f{A})\geq \gamma$, we see that $\delta(\f{B})\geq -(k+1)\gamma$ and hence $\delta(\f{D'})$. Thus $\f{D'}\in\Kfin$. 	 
	
	Now we show that $(\f{A,D})$ is an essential minimal pair with $0>\delta(\f{D/A})\geq -\gamma$. So assume that $\f{A\subseteq{D'}\subsetneq{D^*}}$. If $\f{B'\neq{B}}$, then $\delta(\f{D'/A})=\delta(\f{D'/B'})+\delta(\f{B'/A})\geq{0}$. So assume that $\f{B'=B}$. Thus $\delta(\f{D'/A})= \delta(\f{D'/B}) + k\gamma$. Since each $(\f{B,C_i})$ is an essential minimal pair, it follows that $\delta(\f{D'/B})\geq -k\gamma$ \textit{unless} $\f{D'=D}$. Thus $(\f{A,D})$ forms an essential minimal pair with the required properties. 	   
\end{proof}

Finally we are in a position to prove one of the key result of this section:

\begin{thm}\label{lem:InfMinPa}\label{thm:OmitPrelim2}
	Let $\f{A}\in\Kfin$ with $\delta(\f{A})>0$. \begin{enumerate}
		\item If $\z\alpha$ is not rational, then for any $\epsilon>0$, we can construct infinitely many non-isomorphic $\f{D}\in\Kfin$ such that $(\f{A,D})$  is an essential minimal pair that satisfies $-\epsilon<\delta(\f{D/A})<0$. 
		
		\item If $\z\alpha$ is rational, then we can construct infinitely many non-isomorphic $\f{D}\in\Kfin$ such that $(\f{A,D})$ is an essential minimal pair that satisfies $\delta(\f{D/A})=-1/c$. (Recall that $c$ denotes the least common multiple of the denominators of the $\z\alpha_E$).  
	\end{enumerate} 
\end{thm}

\begin{proof}
	
	For $|A|\geq m_\text{suff}$, the required results are immediate from Lemma \ref{thm:OmitPrelim3}. So assume that $|A|<m_\text{suff}$. Let $\f{A}_0$ be an $L$-structure with $m_\text{suff}$ many points such that no relations hold on $A_0$ and take $\f{B}=\f{A}\oplus{\f{A}}_0$. Clearly $\f{A\leq B}$. Using Theorem \ref{thm:OmitPrelim3} fix a $\f{C}$ such that $(\f{B,C})$ is an essential minimal pair $\f{C}\in\Kfin$. Note that if $\z\alpha_E$ is irrational for some $E\in L$ and $\epsilon>0$, then we may assume that $-\min\{\epsilon,\delta(\f{A})\}<\delta(\f{C/B})<0$ and if $\z\alpha$ is rational, then we may assume $\delta(\f{C/B})=-1/c$. By using $(2)$ of Lemma \ref{lem:NewStrsViaFreeJoins}, we obtain a required structure $\f{D}$. We observe that the non-isomorphic $\f{D}$ may be obtained by varying our choice of $\f{C}$ and leave it to the reader to verify that in the case $\z\alpha$ is rational, we have $\delta(\f{D/A})=-1/c$ as claimed. 	  
\end{proof}

\subsection{Coherence and rank $0$ structures}\label{subsec:Coherence}

This section is dedicated to building finite extensions of rank $0$. Our goal is to show that if $\z\alpha$ is coherent, then for any $\f{B}\in\Kfin$ with $\delta(\f{B})>0$, there is some $\f{D}\in\Kfin$ with $\f{B\subseteq D}$ such that $\delta(\f{D})=0$. If $\z\alpha$ is rational, this is easily achieved by repeated use of $(2)$ of Theorem \ref{lem:InfMinPa}. Thus we focus on the case that $\z\alpha$ is coherent but not rational. 

\begin{defn}\label{defn:softbound}
	Let $\z\alpha$ be coherent but not rational. Let $\beta(\z\alpha)=\min\{\delta(\f{A}), Gr(2):\f{A}\in\Kfin, \delta(\f{A})>0 \text{ and } |A|<m_\text{suff} \}$. 	
\end{defn}

\begin{remark}\label{rmk:ClarOfSoftBound}
	Note that $\beta(\z\alpha)>0$. Further if $\f{B}\in\Kfin$ is such that  $0<\delta(\f{B})<\beta(\z\alpha)$, then $|B|\geq m_\text{suff}$.
\end{remark}

\begin{prop}\label{prop:justifyZeroSet}
	Let $\f{B}\in\Kfin$. Then there is some $\f{Z}\subseteq B$ such that $\delta(\f{Z})=0$ and if $\f{C}\subseteq\f{B}$ is such that $\delta(\f{C})=0$, then $\f{C\subseteq Z}$.
\end{prop}

\begin{proof}
	Let $\f{B}\in\Kfin$ and let $\f{A,C}\subseteq{B}$ with $\delta(\f{A})=\delta(\f{C})=0$. Let $\f{D}$ be the join of $\f{A,C}$ in $\f{B}$. Now $0\leq\delta(\f{D})\leq \delta(\f{A})+\delta(\f{C})=0$ by $(3)$ of Fact \ref{lem:SubModularityOverBase}. Thus there is a unique maximal (with respect to $\subseteq$) $\f{Z\subseteq B}$ such that $\delta(\f{Z})=0$.
\end{proof}

\begin{defn}\label{defn:ZeroSet}
	Let $\f{B}\in\Kfin$. The unique maximal (with respect to $\subseteq$) $\f{Z\subseteq B}$ such that $\delta(\f{Z})=0$ will be called the \textit{zero set of} $\f{B}$ and we denote $\f{Z}$ by $\f{Z}_\f{B}$. We will let $Z_B$ denote the universe of $\f{Z_B}$.   
\end{defn}

\begin{lemma}\label{lem:ZeroExtForSmallRank}
	Let ${\z\alpha}$ be coherent and assume that $\z\alpha$ is not rational. Let $\f{A}\in\Kfin$ with $\beta(\z\alpha)>\delta(\f{A})>0$. Then there exists $\f{A^*}\in\Kfin$ such that $\f{A^*\supseteq A}$, $0\leq \delta(\f{A}^*)<\beta(\z\alpha)$ and $|A^*-Z_{A^*}|<|A-Z_A|$.
\end{lemma}

\begin{proof}
	Choose $\f{B}\subseteq \f{A}$ such that  $\f{Z}_\f{A}\subsetneq\f{B}\subseteq\f{A}$ and  $\gamma:=\delta(\f{B})$ is least possible. Clearly $\gamma>0$ as $\f{Z_A\subsetneq B}$, $\f{B\leq A}$ as the rank of $\f{B}$ is minimal and $|B|\geq m_\text{suff}$ as $\gamma\leq \delta(\f{A})<\beta(\z\alpha)$. Further using $(2)$ of Fact \ref{lem:BasicDel1}, it follows that for any $\f{B'\subseteq B}$, either $\f{B'\subseteq Z_A}$ or $\delta(\f{B'})\geq\gamma$. We construct $\f{A^*}$ by taking a free join of $\f{A}$ over $\f{B}$ with a suitably constructed structure $\f{D}\in\Kfin$ with $\f{B\subseteq D}$.   
	
	%Proof of last line in claim: Let $\f{B'\subseteq {B}}$. It suffices to consider $\f{B'\nsubseteq Z_A}$. Note that if $\f{Z_A\subsetneq B'\subseteq B}$, then $\delta(\f{B'})\geq \gamma$ by our choice of $\f{B}$. So assume that $\f{Z_A\nsubseteq B'}$. Using $(2)$ of Theorem \ref{lem:BasicDel1}, we see that $\delta(\f{B'/Z_A\cap{B'}})\geq\delta({B'Z_A/Z_A})=\delta(B'Z_A)\geq\gamma$. As $\delta(\f{B'})\geq \delta(\f{B'/Z_A\cap{B'}})$, we see that $\delta(\f{B'})\geq \delta(\f{A})$. 
	
	Now as ${\z\alpha}$ is coherent there are infinitely many positive integers $ \langle n',m_E'\rangle_{E\in L}$ such that $n'-\sum_{E\in L}m_E'\z\alpha_E=0$. Using the fact that $\gamma=\delta(\f{C})$, we obtain that $\delta(\f{C})=n_0-\sum_{E\in L}m_0(E)\z\alpha_E$ for some non-negative integers $\langle n_0,m_0(E)\rangle_{E\in L}$. Hence we now obtain that there are infinitely many positive integers $\langle n'',m_E''\rangle_{E\in L}$ such that $n''-\sum_{E\in L}m_E''\z\alpha_E=-\gamma$. Thus we can construct acceptable $\langle n, s \rangle$ such that $w(n,s)=-\gamma$. Use Lemma \ref{lem:preOmitPrelim2} to construct an $n$-template $\Theta$ that corresponds to $\langle n,s \rangle$. 
	
	Fix any $b^*\in B-Z_{A}$. Let $\f{D}$ be an extension of $\f{B}$ by $\Theta$ with the additional property that there is some relation $E$ and $Q\in E^{\f{D}}$ with $\{b^*,d_n\}\subseteq Q$ where $d_n$ is as described in Notation \ref{notation:ThetaExtension}. As $\delta(\f{D/B})=-\gamma$ we have that $\delta(\f{D})=0$. We claim that $\f{D}\in\Kfin$. %Let $\f{D'\subseteq D}$ and let $\f{B'=B\cap D}$. It suffices to show that $\delta(\f{D'})\geq 0$.
	
	First note that if $\f{B\subseteq D'\subseteq D}$, then $\delta(\f{D'/B})\geq -\gamma$ by $(1.c)$  of Lemma \ref{lem:preOmitPrelim}.  Hence we obtain that $\delta(\f{D'})\geq 0$. Now choose $\f{D'\subseteq D}$ arbitrary and and let $\f{B'=B\cap D'}$. There are now three possibilities. First consider the case $d_n\notin\f{D'}$. By $(1.b)$ of Lemma \ref{lem:preOmitPrelim} we obtain that $\delta(\f{D'/B'})\geq 0$ and hence we obtain that $\delta(\f{D'})\geq 0$ as $\f{B'}\in\Kfin$. Now consider the case $b^*\in\f{D'}$. Then we have that $b^*\in B'$ and hence $\delta(\f{B'})\geq \gamma$. As $\delta(\f{D'/B'})\geq \delta(BD'/B)$ by $(2)$ of Fact \ref{lem:BasicDel1} and $\delta(BD'/B)\geq -\gamma$, we conclude that $\delta(\f{D'})\geq 0$. Finally consider the case $d_n\in D'$ but $b^*\notin D'$. Then we have that $Q\notin E^{\f{D'}}$. So $\delta(\f{D'/B'})\geq \delta(BD'/B)+\z\alpha(E)\geq 0$. As $\delta(B')\geq 0$, $\delta(\f{D'})\geq 0$.
	
	%%% DO NOT DELETE!!! %%% This is a justification for the first line of the above, using Notation \ref{notation:ThetaExtension}: Note that for $j<n$, $\delta(\f{D^j/B})\geq 0$ and $\z\alpha(E_j)-\delta(\f{D}_j/\f{B})\geq 0$ as $\Theta$ is acceptable. Further $\delta(\f{D}^{k+1,j}/\f{B})=\delta(\f{D}^j/\f{B})+\z\alpha(E_k)-\delta(\f{D}^k/\f{B})$ for $1\leq k <j\leq n$ and $-\gamma=\delta(\f{D}^j/\f{B})<0$ if and only if $j=n$. Since $\f{D'}$ can be written as the free join of several $\f{D}^{k,j}$ and over $\f{B}$ and at most one of the $\f{D}^{k,j}$ satisfies $0>\delta(\f{D}^{k,j}/\f{B})\geq -\gamma$, it follows that $\delta(\f{D'/B})\geq -\gamma$.  
	
	Let $\f{A^*}$ be the free join $\f{D\oplus_{B}A}$. As $\f{B\leq A}$ and $\f{D}\in\Kfin$, by Fact \ref{lem:FullAmalg1}, we obtain that $\f{A^*}\in\Kfin$. Now $\delta(\f{A^*/B})=\delta(\f{A/B})+\delta(\f{D/B})=\delta(\f{A/B})-\gamma$ and hence $0\leq \delta(\f{A}^*)<\beta(\z\alpha)$.       
	
	Finally note that the universe of $\f{A^*}$ is $A\cup D$.  As $\delta(\f{D})=0$, we have that $\f{B\subseteq D\subseteq Z_{A^*}}$. As $b^*\in B-Z_{A}$, we conclude that $|A^*-Z_{A^*}|<|A-Z_A|$.    
\end{proof}

\begin{thm}\label{thm:AnnhilConstruc}
	Let ${\z\alpha}$ be coherent. Then given any $\f{A}\in{\Kfin}$ with $\delta(\f{A})>0$ there is   $\f{D}\in\Kfin$ such that $\f{D\supseteq A}$ and $\delta(\f{D})=0$.
\end{thm}

\begin{proof}
	\noindent\textit{Case 1}: Assume that $\z\alpha$ is not rational. Now there is some $E\in L$ such that $\z\alpha_{E}$ is irrational. If $0\leq\delta(\f{A})<\beta(\z\alpha)$, then we are done. So assume that $\delta(\f{A})\geq{\beta(\z\alpha)}$. Since $\z\alpha_{E}$ is irrational, we can find a  minimal pair $(\f{A,B})$ with $\delta(\f{B/A})$ as small as we like using Theorem \ref{lem:InfMinPa}. Now fixing a minimal pair such that $\delta(\f{B/A})<\beta(\z\alpha)$ and taking sufficiently many isomorphic copies of $\f{B}$ freely joined over $\f{A}$, we can find a $\f{A^*\supseteq A}$ such that $\f{A^*}\in\Kfin$ and $0<\delta(\f{A^*})<\beta(\z\alpha)$. Let $l=|A^*-Z_{A^*}|$. By iterating Lemma \ref{lem:ZeroExtForSmallRank} at most $l$ times, we may construct $\f{D\supseteq A^*}$ with $\f{D}\in\Kfin$ such that $|D-Z_{D}|=0$, i.e. $\delta(\f{D})=0$.\\
	
	\noindent\textit{Case 2}: Assume that $\z\alpha$ is rational. Then $\delta(\f{A})=k/c$ for some positive integer $k$, where $c$ is the least common multiple the $q_E$ where $\z\alpha_E=p_E/q_E$ (in reduced form). As noted in Theorem \ref{lem:InfMinPa} we may create a minimal pair $\f{B}$ over $\f{A}$ such that $\delta(\f{B/A})=-1/c$ and for all $\f{B'}\subsetneq{B}$, $\delta(\f{B'/A\cap{B'}})\geq{0}$. Let $\f{D}=\oplus_{1\leq i\leq k}\f{B}_i/\f{A}$, the free join of $k$ isomorphic copies of $\f{B}$ over $\f{A}$. A routine argument now shows that  $\delta(\f{D})=0$ and that $\f{D}\in{\Kfin}$.
\end{proof}

\begin{remark}
	We note that we may construct infinitely many such non-isomorphic $\f{D}$ by varying our choice of $\f{A^*}$ or $\f{B}$ accordingly.  
\end{remark} 

%%% CAN THIS BE DELETED? %We begin with the following lemma captures a key property of extensions that cover the base over which the construction is being carried out. 

%\begin{lemma}\label{lem:RemovalOfPoint}
%	Let $\f{B}\in\Kfin$ be non-empty and let $\Theta$ be an $n$-template. Let $\f{D}$ be an extension of $\f{B}$ by $\Theta$ that covers $\f{B}$. Let $\f{D'\subsetneq D}$ and $\f{B'}=\f{B\cap D}$ with $D'-B=D-B$. Now $\delta(\f{D'}/\f{B}')\geq \delta(\f{D/B})+\z\alpha(E)$ where $E$ is some relation that contains the point $b\in B-B'$ and a point $d\in D-B$.  
%\end{lemma} 

%\begin{proof}
%	First recall our notation of $e(A), e(A,B)$ from Notation \ref{notation:e(A)}. We note that $\delta(\f{D'/B'})=\delta(D'-B')-e(D'-B',B')$ and $\delta(\f{D/B})= \delta(D-B)-e(D-B,B)$ by using $(1)$ of \ref{lem:BasicDel3}. Let $E$ be some relation that contains the point $b\in B-B'$ and $d\in D-B$. Such an $E$ exists as $\f{D}$ covers $\f{B}.$ It follows that $e(D-B,B)\geq e(D-B',B')+\z\alpha(E)$. Further $D'-B=D-B=D'-B'$. Hence $\delta(\f{D/B})\leq \delta(D'-B')-e(D'-B',B')-\z\alpha(E)$ and the result follows. 
%\end{proof}

\section{Quantifier elimination and the completeness of $\Sfin$}\label{sec:QuantElim}

In this section we begin by introducing a collection of $\forall\exists$-axioms that we denote by $\Sfin$ (see Definition \ref{defn:AEAxioms}). In Theorem \ref{thm:QuantElim} we observe that $\Sfin$ admits quantifier elimination down to the level of chain minimal extension formulas (see Definition \ref{defn:ExtensionFormula}).  This generalizes the results of Laskowski in \cite{Las1}. In Theorem \ref{Cor:SfinComplete} we collect useful results about $\Sfin$ including the fact that $\Sfin$ is the theory for the Baldwin-Shi hypergraph for $\z\alpha$. Lemma \ref{lem:CompleteTypesOverSets} gathers useful consequences of the quantifier elimination. Remark \ref{rmk:DOP} is out of character with the rest of this paper: we sketch a proof of the dimensional order property for $\Sfin$, again following ideas found in \cite{Las1}.

\begin{defn}\label{defn:AEAxioms}
	The theory $\Sfin$ is the smallest set of sentences insuring that if $\f{M}\models{\Sfin}$, then
	\begin{enumerate}
		\item $\f{M}\in{\z{\Kfin}}$, i.e. every finite substructure of $\f{M}$ is in $\Kfin$ 
		\item For all $\f{A\leq{B}}$ from $\Kfin$, every (isomorphic) embedding $f:\f{A\rightarrow{M}}$ extends to an embedding $g:\f{B\rightarrow{M}}$
	\end{enumerate}
\end{defn}

\begin{remark}\label{rmk:SfinIsAEAxiomatizable}
	We note that $\Sfin$ is a collection of $\forall\exists$-sentences. 
\end{remark}

\begin{notation} Let $\f{N}\in\z{K_L}$. Given  $\f{A}\in{K_L}$ with a fixed enumeration $\z{a}$ of $A$, we write $\Delta_\f{A}(\z{x})$ for the atomic diagram of $\f{A}$. Also for $\f{A},\f{B},\f{C}\in K_L$ with $\f{A}\subseteq{\f{B}}\subseteq{\f{C}}$ and fixed enumerations $\z{a}, \z{b}, \z{c}$ respectively with $\z{a}$ an initial segment of $\z{b}$ and $\z{b}$ an initial segment of $\z{c}$;  we let $\Delta_{\f{A,B}}(\z{x}, \z{y})$ the atomic diagram of $\f{B}$ with the universe of $\f{A}$ enumerated first according to the enumeration $\z{a}$. Similarly $\Delta_{\f{A,B,C}}(\z{x}, \z{y}, \z{z})$ will denote the atomic diagram of $\f{C}$ with the universe of $\f{A}$ enumerated first by $\z{x}$, the remainder $B-A$ by $\z{y}$ and then $C-B$ by $\z{z}$ according to the enumerations $\z{a}, \z{b}, \z{c}$.  
\end{notation}

\begin{defn}\label{defn:ExtensionFormula}
	Let $\f{A,B}\in{K}$ and assume $\f{A\subseteq{B}}$. Let $\Psi_{\f{A,B}}(\z{x})=\Delta_\f{A}(\z{x})\wedge\exists{\z{y}}\Delta_{\f{(A,B)}}{(\z{x},\z{y})}$. Such formulas are collectively called \textit{extension formulas} (over $\f{A}$). A \textit{chain minimal extension formula} is an extension formula $\Psi_{\f{A,B}}$ where $\f{B}$ us the union of a minimal chain over $\f{A}$. 
\end{defn}

\subsection{Some Preliminaries}\label{subsec:ImmediateConseqOfOmit}

This Section contains several Lemmas that will be needed in the proof of the quantifier elimination result of $\ref{thm:QuantElim}$. We begin by generalizing Proposition 4.2 of \cite{Las1}. Recall that if $\z\alpha$ is not rational, then $\lim_n Gr(n)=0$. Thus in the case $\z\alpha$ is not rational we may replace clause $(1)$ of the following lemma with $0\leq\delta(\f{D^*/A})<\mu$ where $\mu>0$. The new statement thus obtained is precisely Proposition 4.2 of \cite{Las1}. 

\begin{lemma}\label{lem:Omit1}
	Suppose that $\f{A\leq{B}}\in{\Kfin}$ and $\Phi\finsubset{\Kfin}$ are given such that $\f{B\subseteq{C}}$ with $\f{B\nleq{C}}$ for all $\f{C}\in\Phi$. Let $m\in{\omega}$. Then there is a $\f{D^*\supseteq{B}}$, $\f{D^*}\in\Kfin$ such that \begin{enumerate}
		\item $0\leq\delta(\f{D^*/A})<Gr(m)$
		\item $\f{A\leq{D^*}}$
		\item No $\f{C}\in\Phi$ isomorphically embeds into $\f{D^*}$ over $\f{B}$
	\end{enumerate}
	If $\z\alpha$ is rational then we can always find $\f{D^*}$ such that $\delta(\f{D^*/A})=0$.  
\end{lemma}

\begin{proof} 
	Fix $\f{A,B}$ and $\Phi$ as above. Note that we may replace each $\f{C}\in\Phi$ by $\f{B\subseteq{C'}\subseteq{C}}$ that is minimal and thus we may as well assume that $\f{(B,C)}$ is a minimal pair for any given $\f{C}\in\Phi$. Now if $\delta(\f{A})=\delta({\f{B}})$, then take $\f{D^*=B}$.  So we may assume that $\delta(\f{A})<\delta(\f{B})$. Let $u$ be a positive integer such that $u>|\f{C}|$ for each $\f{C}\in{\Phi}$. Now using Theorem \ref{thm:OmitPrelim2}, fix a $\f{D}\in\Kfin$ such that $|D-B|>u$ and $(\f{B,D})$ is an essential minimal pair that satisfies $-\min\{Gr(m),\delta(\f{B/A})\}\leq \delta(\f{D/B})<0$. Using $(1)$ of Lemma \ref{lem:NewStrsViaFreeJoins}, we may obtain $\f{D}^*$ with the required properties. 
\end{proof}

\begin{defn}
	Let $\f{B}\in{\Kfin}$ and let $\Phi\finsubset{\Kfin}$ such that each $\f{C}\in{\Phi}$ extends $\f{B}$. For any $\f{M}\models{S_{\alpha}}$, an embedding $g:\f{B}\rightarrow\f{M}$ \textit{omits} $\Phi$ if there is no embedding $h:\f{C}\rightarrow\f{M}$ extending $g$ for any $\f{C}\in{\Phi}$.
\end{defn}

The following is a Proposition 4.4 of \cite{Las1}. It's proof follows along the same lines there in with obvious modifications made to allow for the existence of structures $\f{D}\in\Kfin$ such that $\delta(\f{D})=0$ in the case that $\z\alpha$ is rational.  

\begin{thm}\label{thm:MainEmThm}
	Suppose that $\f{A\leq{B}}$ are from $\Kfin$ and $\Phi$ is a finite subset of of $\Kfin$ such that for each $\f{C}\in{\Phi}$, $\f{A\leq{C}}$, $\f{B\subseteq{C}}$ but $\f{B\nleq{\f{C}}}$. Then for any $\f{M}\models{\Sfin}$, for any embedding $f:\f{A\rightarrow{M}}$ there are infinitely many embeddings $g_i:\f{B\rightarrow{M}}$ extending $f$ such that each $g_i$ omits $\Phi$ and $\{g_{i}(\f{B}):i\in{\omega}\}$ is disjoint over $f(\f{A})$. 
\end{thm}

\begin{cor}\label{lem:aleph0sat}
	Suppose that $\f{A,B}\in\Kfin$ and $\f{A\leq{B}}$ and $f:\f{A\rightarrow{M^*}}$ is strong where $\f{M}^*\models{\Sfin}$ is $\aleph_{0}$-saturated. Then there is a strong embedding $g:\f{B\rightarrow{M^*}}$ extending $f$. In particular, every $\f{B}\in{K_{\alpha}}$ embeds strongly into $\f{M}^*$.
\end{cor}

\subsection{Putting it all together}\label{subsec:QuantElim}

In this section we give a brief description of how to genaralize the results of \cite{Las1} mentioned at the beginning of this section. \\

Suppose that $\f{A\subseteq{B}}$ are from $\Kfin$. Let $\f{C}$ be the union of a maximal minimal chain of minimal pairs over $\f{A}$ in $\f{B}$. Then clearly $\f{C\leq{B}}$.  Since the sentence $\forall{\z{x}}[\Delta_\f{C}(\z{x})\rightarrow\exists{\z{y}}\Delta_{\f{(C,B)}}{(\z{x},\z{y})}]$ is an axiom $\Sfin$, the extension formula $\Psi_{\f{A,B}}$ is $\Sfin$ equivalent to the chain-minimal extension formula $\Psi_{\f{A,C}}$, i.e. every extension formula is $\Sfin$ equivalent to a chain minimal extension formula. 

\begin{thm}\label{thm:QuantElim}
	Every $L$-formula is $\Sfin$-equivalent to a boolean combination of chain-minimal extension formulas.
\end{thm}

\begin{proof}
	The proof is identical to the proof of Theorem 5.6 of \cite{Las1}. The proof in \cite{Las1} depends on results in Section 3 and Proposition 4.4 of \cite{Las1}. As we have noted previously, Theorem \ref{thm:MainEmThm} generalizes Proposition 4.4 of \cite{Las1}. The results in Section 3 of \cite{Las1}  are easily seen to hold in this context. 
\end{proof}

Of the following results, $(1)$ and $(2)$ of Theorem \ref{Cor:GenModelsSfin} was first proved in full generality in \cite{IkKiTs} by Ikeda, Kikyo and Tsuboi. However their proof does not yield the quantifier elimination result of Theorem \ref{thm:QuantElim}. See Corollary 5.7 and Proposition 6.1 of \cite{Las1} for an alternate proof of Theorem \ref{Cor:GenModelsSfin} using the techniques found in this paper. 

\begin{theorem}\label{Cor:SfinComplete}\label{Cor:GenModelsSfin}\label{lem:ClosEqui}
	\begin{enumerate}
		%\item The theory $\Sfin$ is complete.
		\item $\Sfin$ is the theory of the $(\Kfin,\leq)$-generic $\f{M}_{\z\alpha}$. 
		\item Fix $\f{M}\models{\Sfin}$ and $\f{X\subseteq{M}}$. The following are equivalent:\begin{enumerate}
			\item $X$ is algebraically closed
			\item For any minimal pair $\f{(B,C)}$ with $\f{C \subseteq M}$, if $B\subseteq{X}$, then $C\subseteq{X}$.
			\item For any finite ${B\subseteq{M}}$, ${B}\cap{X}\leq{{B}}$ 
		\end{enumerate}
	\end{enumerate}
\end{theorem}

The following lemma, will be useful in both Section \ref{sec:AtomECMod}. It is an immediate consequence of the quantifier elimination:

\begin{lemma}\label{lem:CompleteTypesOverSets}
	Let $\f{M}\models\Sfin$ and ${A}$ be a finite closed set of $\f{M}$. Suppose that $\pi$ is a consistent partial type over $A$. Then \begin{enumerate}
		\item If $\f{M}$ is $\aleph_0$-saturated and any realization $\z{b}$ of $\pi$ in $\f{M}$ has the property that  $\z{b}{A}$ is closed in $\f{M}$, then $\pi$ has a unique completion to a complete type $p$ over $A$.
		\item If any realization $\z{b}$ of the quantifier free type of $\pi$ (over ${A}$) has the property $\delta(\z{b}/{A})=0$, then $\pi$ has a unique completion $p$ over $A$ and further $p$ is isolated by the formula $\Delta_{{A},A\z{b}}(\z{a},\z{x})$. 
	\end{enumerate}  
\end{lemma}

\begin{proof}
	$(1)$: Note that by Theorem \ref{thm:QuantElim} it suffices to show that all chain minimal formulas over ${A}$ are determined by the given conditions. Let $\z{b}\models \pi$. Fix $\z{b}{A}\subseteq{D}\in\Kfin$ and let $\phi_{{D}}(\z{x}) = \Delta_{\z{a},\z{a}\z{b}}(\z{a},\z{x})\wedge \exists{\z{y}\Delta_{\z{a},\z{a}\z{b},{D-\z{a}\z{b}}}(\z{a},\z{x},\z{y})}$ be the corresponding extension formula. Suppose that $\z{b}A\leq D$. Now as $\z{b}A\leq M$ and $\f{M}\models\Sfin$, we obtain that $\f{M}\models \phi_{D}(\z{b})$. Thus it follows that $p\proves \phi_{D}$. Now suppose that $\z{b}{A}\nleq{{D}}$. If $\pi^* = \pi\cup{\neg\phi_{D}(\z{x})}$ is consistent, then there is some realization of $\pi^*$ \textit{in} $\f{M}$ by $\aleph_0$-saturation. Clearly no realization of $\pi^*$ can be strong in $\f{M}$, and hence $\pi\proves \neg\phi_{D}(\z{x})$. Thus $\pi$ determines all \textit{extension formulas} including the chain minimal formulas over $A$ and thus is complete. So simply take $p=\pi$ to obtain the required complete type.\\
	
	\noindent $(2)$: Consider a partial type given as above. We may as well assume that $\Delta_{A,A\z{b}}(\z{a},\z{x})\in \pi$. Arguing as in part $(1)$, we see that if $\z{b}A\leq D$, then $\phi_{D}(\z{x})\in \pi$. So assume that $\z{b}A\nleq D$ and that $\neg\phi_{D}(\z{x})$ is consistent with $\pi$. As $\f{M}$ is a model, there is some $\z{b'}$ realizing $\phi_{D}(\z{x})$. But then, there is some $C\subseteq M$ such that $(\z{b}A,\z{b}AC)$ is a minimal pair. Now $\delta(\z{b}AC/A)= \delta(\z{b}AC/\z{b}A)+\delta(\z{b}A/A)<0$. But this contradicts $A\leq M$. Thus the required result follows.
\end{proof}

We take this opportunity to give a brief sketch of the fact that $\Sfin$ has \textit{Dimensional Order Property} (DOP, see \cite{Bd} for a definition). This result will is independent from the rest of the results in the note.   

\begin{remark}\label{rmk:DOP}
	It is well known that these theories are stable (see \cite{BdSh}, \cite{VY}). In \cite{BdShelah}, Baldwin and Shelah gave a proof that $\Sfin$ has DOP \textit{assuming} that $L$  has a binary relation. In Corollary $7.10$ of \cite{Las1}, Laskowski gave a proof of DOP by explicitly constructing a type that witnesses the DOP. He did not assume that $L$ contained a binary symbol, however he did assume  $\z\alpha$ satisfied certain properties. His proof contains two key steps: Proposition 7.8 and Corollary 7.10 of \cite{Las1}. We observe that we can prove a slightly modified form of Proposition 7.8 of \cite{Las1} by replacing $\f{A\leq B}$ but $\f{A\neq B}$ in its statement with $\f{A\leq B}$ but $\delta(\f{A})\neq \delta(\f{B})$ using Lemma \ref{lem:Omit1}. The proof of Corollary $7.10$ will remain unchanged from \cite{Las1}, establishing DOP for $\Sfin$.
\end{remark}

\section{Atomic Models of $\Sfin$ }\label{sec:AtomECMod}

In this section we study the atomic models of the theories of Baldwin-Shi hypergraphs. Our main results begin with Theorem \ref{thm:ClassAtModMain}, in which we characterize the atomic models as the existentially closed models of $\Sfin$ with \textit{finite closures} (see Definition \ref{defn:FinClo}) or equivalently those with finite closures where the closed finite substructures are those with rank $0$. This immediately yields coherence of $\z\alpha$ as a necessary condition for the existence of atomic models for $\Sfin$. We then proceed to combine the results in Section \ref{subsec:Coherence} and chain arguments to obtain Theorem \ref{thm:ExistenceOfAtomic} which establishes coherence of $\z\alpha$ is also sufficient for the existence of atomic models. We also explore the effect that rationality of $\z\alpha$, arguably the most natural form of coherence, has on atomic models of $\Sfin$. Our exploration leads to Theorem \ref{thm:RationalAlphaAndCoherence} which allows us to categorize rational $\z\alpha$ as precisely the coherent $\z\alpha$ with theories of Baldwin-Shi hypergraphs whose models isomorphically embed into an atomic model of the same cardinality. We begin with the following definitions.

\begin{defn}\label{defn:UniversalSent}
	We use $\Sfin^\forall$ to denote the set of universal sentences of $\Sfin$. Note that an $L$-structure $\f{M}$ models $\Sfin$ if and only if $\f{M}\in \z{K_L}$. 
\end{defn} 

\begin{defn}\label{defn:FinClo}
	Given $\f{M}\models\Sfin^{\forall}$, we say that $\f{M}$ has finite closures if for all $\f{A\finsubset M}$, there is some finite $\f{B}\leq \f{M}$ with $\f{A\leq M}$. 
\end{defn}

\begin{defn}\label{defn:dFunc}
	Let $\f{M}\models \Sfin^\forall$. By $d_{\f{M}}$ we denote the function $d_{\f{M}}:\{\f{A}:\f{A}\finsubset\f{M}\}:\rightarrow \mathbb{R}$ such $d_{\f{M}}(\f{A})=\inf \{\delta(\f{B}):\f{A\subseteq B}, \f{B} \text{ finite and } \f{B\leq M} \}$.
\end{defn}     

Our starting point is the following theorem due to Laskowski (Theorem 6.5 of \cite{Las1}). Its proof only uses the quantifier elimination result of Theorem \ref{thm:QuantElim} and thus holds in our generalized context.

\begin{thm}\label{thm:ClassECMod1}
	Let $\f{M}\models{\Sfin}$. Now $d_{\f{M}}(\f{A})=0$ for all finite $\f{A\subseteq{M}}$ if and only if $\f{M}$ is an e.c. model.
\end{thm}

\subsection{Atomic Models}\label{subsec:AtomMod}

Our goal in this section is to prove Theorem \ref{thm:ClassAtModMain}. We begin with the following: 

\begin{remark}
	Given a countable model $\f{M}\models{\Sfin}$, $\f{M}$ has finite closures if and only if $\f{M}$ is the union of a strong chain $\langle \f{A}_i : i\in\omega \rangle$ of elements of $\Kfin$. 
\end{remark}

\begin{lemma}\label{lem:PrinFla}
	Let $\f{M}\models{\Sfin}$ and ${A\finsubset{M}}$ with $\delta({A})=0$. Let $\z{a}$ be a fixed enumeration of ${A}$. Then ${A\leq \f{M}}$ and $\Delta_{{A}}(\z{x})$ isolates the $\tp(\z{a})$ in $\f{M}$.
\end{lemma}

\begin{proof}
	This follows from an application of Lemma \ref{lem:CompleteTypesOverSets}, by noting that $\emptyset\leq\f{M}$ and $\delta({A}/\emptyset)=0$.
\end{proof}

\begin{lemma}\label{lem:ClassAtMod1}
	Let $\f{M}\models{\Sfin}$ be atomic. Let $A\finsubset M$. Now ${A\leq\f{M}}$ if and only if $\delta({A})=0$. Further $\f{M}$ has finite closures.
\end{lemma}

\begin{proof}
	Let $A,\f{M}$ be as stated above. Clearly if $\delta({A})=0$, then ${A\leq\f{M}}$. Now suppose that ${A}\leq{\f{M}}$. If $\delta({A})>0$, then by Theorem \ref{lem:InfMinPa} there are infinitely many non-isomorphic ${C}$ with $({A,C})$ a minimal pair. It follows that no single chain minimal formula, or indeed a boolean combination of chain minimal formulas can rule out the existence of all of these minimal pairs over $A$ as the sentences of $\Sfin$ dealing with extensions only insist upon the existence of \textit{strong} extensions. Since any formula is equivalent to a boolean combination of chain minimal formulas this contradicts the fact that the model is atomic and hence $tp(A/\emptyset)$ is isolated. Thus $\delta(A)=0$. 
	
	We claim that $\f{M}$ has finite closures. Assume to the contrary that $\f{M}$ does not have finite closures. Let $\f{A}\finsubset \f{M}$ be such that there is no finite  $\f{C\leq M}$ such that $\f{A\subseteq M}$. It now follows that there is a $\subseteq$ increasing sequence  $\{\f{A}_i:i\in\omega,\f{A}_i\subseteq \f{M} \text{ such that } \f{A}_0=\f{A} \text{ and each } (\f{A}_i,\f{A}_{i+1}) \text{ is a minimal pair}\}$. Using the downward Lowenhiem Skolem Theorem, we may construct a countable $\f{M'\preccurlyeq M}$ such that $\bigcup_{i<\omega}A_i\subseteq M'$. Note that $M'$ is a countable, atomic and hence prime model of $\Sfin$. We may as well assume that $\f{M'}\preccurlyeq\f{M^*}$ for notational convenience. Recall that $\f{M^*}$ has finite closure and let $\f{A\subseteq C\leq M^*}$ where $|C|$ is finite. Let $i$ be the least integer such that $\f{A}_i\nsubseteq \f{C}$. Clearly $i\geq 1$ and $C\neq A_{i-1}$ (for if $A_{i-1}=C$, then $A_{i}$ is a minimal pair over $C$, which contradicts $C\leq M^*$). Now $C\leq CA_i$ as $C\leq M^*$ and ${A}_i\subseteq CA_{i}$. By using Fact \ref{fact:ExtSmthFraClass} we obtain that $C\cap{A}_i\leq A_i$. Further $A_{i-1}\subseteq C\cap A_{i}\subsetneq A_{i}$ as $A_{i}\not\subseteq C$. But then $A_{i-1}\leq C\cap {A}_{i}$ as $(A_i,A_{i+1})$ is a minimal pair. By the transitivity of $\leq$ we then obtain $A_{i-1}\leq A_{i}$, a contradiction that shows $\f{M'}$ has finite closures.
\end{proof}

\begin{lemma}\label{lem:ClassAtMod2}
	Let $\f{M}\models{\Sfin}$. Assume that $d_{\f{M}}(\f{A})=0$ for all finite $\f{A\subseteq{M}}$ and that $\f{M}$ has finite closures. Then $\f{M}$ is atomic.
\end{lemma}

\begin{proof}
	Let $\f{A\subseteq{M}}$. We begin by fixing an enumeration $\z{a}$ of $\f{A}$. Let $\icl_{\f{M}}{\f{(A)}}=\f{C}$. As $\f{M}$ has finite closures, it follows that $\f{C}$ is finite. It is clear that $d_{\f{M}}(\f{A})=d_{\f{M}}(\f{C})=\delta(\f{C})=0$. Note that if $\f{A=C}$ then we have already established the result and that if $\f{A\neq C}$, then there is no $\f{A\subseteq B\subsetneq C}$ such that $\delta(\f{B})=0$. We claim that the formula $\Psi_{{A,C}}(\z{x})=\Delta_{{A}}(\z{x})\wedge\exists{\z{y}}\Delta_{A,C}(\z{x}, \z{y})$ isolates $\tp(\z{a})$. Now it suffices to show that $\Psi_{{A,C}}(\z{x})$ decides the chain minimal extension formulas.  
	
	Let $\f{M}'\models\Sfin$ and assume that $\f{A'\subseteq M'}$. Let $\z{a'}$ be a fixed enumeration of $\f{A'}$ and assume that $\f{M'}\models\Psi_{A,C}(\z{a'})$. Let $\f{A'\subseteq C'}\subseteq\f{M}'$ and $\z{c'}$ be an enumeration of $C'-A'$ such that $\f{M'}\models \Delta_{{A}}(\z{a'})\wedge\Delta_{A,C}(\z{a'}, \z{c'})$. Note that ${C'\leq M'}$ as $\delta(\f{C'})=0$. Now given a chain of minimal pairs $\f{A'}=\f{B}_0\subseteq\ldots\subseteq\f{B}_n\subseteq \f{M'}$, we have that  $\f{B}_n\subseteq\f{C}'$ as $\f{C}'$ is closed in $\f{M'}$. Thus $\Psi_{{A,C}}(\z{x})$ decides all chain minimal extension formulas thus isolates the type of $\f{A}$.  
\end{proof}

We now obtain the following theorem: 

\begin{thm}\label{thm:ClassAtModMain}
	Let $\f{M}\models{\Sfin}$. The following are equivalent \begin{enumerate}
		\item $\f{M}$ is atomic
		\item $d_{\f{M}}(\f{A})=0$ for all finite $\f{A\subseteq{M}}$ and $\f{M}$ has finite closures.
		\item $\f{M}$ is existentially closed and has finite closures.
		\item For any $\f{A\subseteq{M}}$ finite, there is $\f{B\supseteq{A}}$ such that $\f{B\subseteq{M}}$, $\f{B}$ is finite and $\delta(\f{B})=0$ 
	\end{enumerate} 
\end{thm}

\begin{proof}
	The equivalence of $(1)$ and $(2)$ is immediate from lemma \ref{lem:ClassAtMod1} and lemma \ref{lem:ClassAtMod2}. The equivalence of $(2)$ and $(3)$ is immediate from Theorem \ref{thm:ClassECMod1}. We now show the equivalence of $(2)$ and $(4)$:
	
	Assume $(2)$. Then take $\f{\icl{(A)}=B}$. Since $\f{M}$ has finite closures, it follows that $\f{B}$ is finite. Since $d_{\f{M}}(\f{A})=0$ it follows that $d_{\f{M}}(\f{A})=\delta(\f{B})=0$ and thus $(4)$ follows.  Now assume $(4)$ holds. Since any $\f{B}$ with $\delta(\f{B})=0$ is strong in $\f{M}$. Now pick a $\f{B'}$ such that $\f{A\subseteq{B'}\subseteq{M}}$ and $\f{B'}$ is finite, $\subseteq$ minimal and $\delta(\f{B'})=0$.  
\end{proof}

\subsection{Existence of atomic models}\label{subsec:PrimeMod}

We begin this section by developing tools to prove Theorem \ref{thm:ExistenceOfAtomic} which establishes that coherence is necessary and sufficient for the existence of atomic models. The proof of sufficiency will involve several steps. The idea is to use the $\forall\exists$-axiomatization of $\Sfin$ to construct atomic models as the union of a chain under $\subseteq$. However, as dictated by Theorem \ref{thm:ClassAtModMain}, atomic models of $\Sfin$ must have finite closures. This introduces the need to carefully keep track of how closures change as you go up along the chain. 

We then proceed to prove Theorem \ref{thm:RationalAlphaAndCoherence} which establishes that for coherent $\z\alpha$, the rationality of $\z\alpha$ is equivalent to every model of $\Sfin$ being isomorphically embeddable in an atomic model of $\Sfin$. A key step in the proof is Lemma \ref{lem:CoherenceWithIrrational}, which constructs a model that does not embed into \textit{any} atomic model by exploiting the fact that there is no decreasing sequence of real numbers of order type $\omega_1$.

\begin{defn}\label{defn:PreservationOfClosures}
	Let $\f{M,N}\models\Sfin^\forall$ with $\f{M\subseteq N}$. We say that $\f{N}$ \textit{preserves closures for} $\f{M}$ if ${X\subseteq M}$ is closed in $M$, then ${X}$ is closed in $N$. 
\end{defn} 

\begin{lemma}\label{lem:ModelsOfUniTh}
	Let $\f{M}\models\Sfin^\forall$ and $\f{A,B}\in{\Kfin}$. Assume that $\f{B\cap{M}=A}$ and let $\f{N=M\oplus_{A}B}$.  \begin{enumerate}
		\item  If $\f{A\leq B}$ or $\f{A\leq M}$, then $\f{N}\models\Sfin^\forall$.
		\item  If $\f{A\leq{B}}$, then $\f{N}$ preserves closures for $\f{M}$ 
		\item If $\f{A\leq{M}}$, then $\f{B}\leq{\f{N}}$
		\item If $\f{A\leq B}$ or $\f{A\leq M}$ and $\f{M}$ has finite closures, then so does $\f{N}$.
	\end{enumerate}  
\end{lemma}

\begin{proof}	
	$(1)$: Assume that $\f{A\leq B}$ or $\f{A\leq M}$. We show that $\f{N}\models\Sfin^\forall$. Note that if not, there is some $\f{A\subseteq C\finsubset{M}}$ such that for some $\f{B'\subseteq{B}}$, $\f{A'\subseteq A}$ and $\f{C'\subseteq C}$, $\f{B'\oplus_{A'}C'}\notin{\Kfin}$. But if this were the case then $\f{B\oplus_{A}C}\notin{\Kfin}$. However we have that $\f{A\leq{C}}$ or $\f{A\leq{B}}$ by our assumption and hence $\f{B\oplus_{A}C}\in{\Kfin}$ by Fact \ref{lem:FullAmalg1}. A contradiction that establishes the our claim. \\
	
	\noindent $(2)$: Assume that $\f{A\leq{B}}$. Let $X\subseteq M$ be closed in $\f{M}$. By way of contradiction assume that $X$ is not closed in $\f{N}$. Thus there is some $D\finsubset X$, $E\finsubset N$, $(D,E)$ is a minimal pair but $E\not\subseteq X$. Let ${A'}=E\cap{A}$, ${B'}=E\cap{(B-A)}$ and ${D'}=E\cap{(D-A)}$. Now note that $0>\delta(E/D)=\delta({B'}/{D'A'})=\delta({B'})-e({B'},{D'A'})\geq\delta({B'})-e({B'},{D'A})={\delta({B'})-e({B'},{A})}\geq{0}$ using $(1)$ of Fact \ref{lem:MonoRelRankOvBase}. Thus it follows that $\f{N}$ preserves closures for $\f{M}$.\\ 
	
	For the proof of $(3), (4)$, first note that if $B\subseteq F\finsubset N$, then we may write $F=B\oplus_{A}F'$ with $F'\subseteq M$. Further if $F\subseteq G\subseteq N$ with $G=B\oplus_{A}G'$, then $\delta(G/F)=\delta(G'/F')$. Also to show that $F\finsubset N$ is strong in $N$, it suffices to show that $\delta(G/F)\geq 0$ for all finite $F\subseteq G\finsubset N$. \\
	
	\noindent $(3)$: Assume that $\f{A\leq{M}}$. Given $B\subseteq G\finsubset N$. Take $F=B\oplus_{A}A=B$ and $G=B\oplus_A G'$ where $G'=G\cap{M}$. Now it follows that $\delta(G/F)=\delta(G'/A)$. Since $\f{A\leq M}$, it follows that $\delta(G'/A)\geq 0$. Thus $\f{B\leq N}$. \\ 
	
	\noindent $(4)$:  Assume that $\f{M}$ has finite closures. We wish to show that $\f{N}$ has finite closures. Let $X\finsubset N$. Since intrinsic closures are monotonic with respect to $\subseteq$, we may as well assume that $B\subseteq X$. Let $F=\icl_{{M}}(X\cap M)$. Note that $F'$ is finite because $\f{M}$ has finite closures. Take $F=B\oplus_A F'$ and note that $X\subseteq F$. Fix $F\subseteq G \finsubset N$ with $G=B\oplus_A G'$ where $G'=G\cap N$. Now $\delta({G/F}) = \delta({G'/F'})$ from which the result follows as $\delta(G'/F')\geq 0$ as $F'\leq M$. 
\end{proof}

\begin{lemma}\label{lem:PresOfClos}
	Let $\langle\f{M}_{\beta}\rangle_{\beta<\kappa}$ be a $\subseteq$-chain of models of $\Sfin^\forall$ with $\f{M}_\gamma=\bigcup_{\beta<\gamma}\f{M}_{\beta}$ for limit $\gamma$. Assume that $\f{M}_{\beta+1}$ preserves closures for $\f{M}_\beta$ for each $\beta<\kappa$. Then $\f{M}=\bigcup_{\beta<\kappa}\f{M}_\beta$ preserves closures for each $\f{M}_\beta$, $\beta<\kappa$. Further if $\f{M}_\beta$ has finite closures for each $\beta<\kappa$, then so does $\f{M}$  
\end{lemma}

\begin{proof}
	Let $\f{M}$ be as above and let $\f{X}\subseteq\f{M}_\beta$ be closed. We claim that if $\f{X}$ is closed in $\f{M}$, then it is closed in $\f{N}$. By way of contradiction, suppose not. Then there is some minimal pair $(\f{A,B})$ with $\f{B\subseteq M}, \f{A\subseteq X}$ and $\f{B\subsetneq X}$ that witnesses this. Let $\gamma>\beta$ be the least ordinal such that $\f{B}\subseteq \f{M}_\gamma$. As closures are preserved for successor ordinals, it follows that $\gamma$ is not a successor ordinal. Thus $\gamma$ must be a limit ordinal. But $\f{M}_\gamma = \bigcup_{\beta<\gamma}\f{M}_{\beta}$ which implies $\f{B}\subseteq \f{M}_{\gamma'}$ for some $\gamma'<\gamma$. But then $\f{X}$ is not closed in $\f{M}_{\gamma'}$, which contradicts the minimality of $\gamma$. Thus the first claim is true. The second claim follows by a similar argument.   
\end{proof}

We now illustrate how to extend a model of the universal sentences of $\Sfin$ to a model of $\Sfin$, while preserving closures, a key step towards building atomic models. 

\begin{lemma}\label{lem:CompletionsOfUniTh1}
	Let $\f{M}\models\Sfin^\forall$ be infinite. There exists $\f{N}\models\Sfin$ such that $\f{M\subseteq{N}}$, $|M|=|N|$, $\f{N}$ preserves closures for $\f{M}$. Further if $\f{M}$ has finite closures, then $\f{N}$ has finite closures too.
\end{lemma}

\begin{proof}
	Let $\f{M}\models\Sfin^\forall$. Fix a finite $\f{A\subseteq M}$. A routine chain argument using Lemma \ref{lem:ModelsOfUniTh} allows us to create $\f{M'}$ with the following properties: \begin{enumerate}
		\item $\f{M'}$ preserves closures for $\f{M}$ and $|M'|=|M|$
		\item If $\f{B}\in\Kfin$ with $\f{A\leq B}$, there is some $g$ that embeds $\f{B}$ into $\f{N}$ over $\f{A}$.
		\item If $\f{B}_1,\f{B}_2\in\Kfin$ with $\f{A}\leq \f{B}_1,\f{B}_2$ and $\f{B}_1,\f{B}_2$ are not isomorphic over $\f{A}$, then there are embeddings $g_1,g_2$ of $\f{B}_1,\f{B}_2$ over $\f{A}$ such that $g_1(\f{B}_1), g_2(\f{B}_2)$ are freely joined over $\f{A}$.
	\end{enumerate} 
	
	Note that $\f{A}$, when considered as a substructure of $\f{M'}$, satisfies the extension formulas required by $\Sfin$. Further, by an application of Lemma \ref{lem:PresOfClos}, it follows that if $\f{M}$ has finite closures, then so does $\f{M'}$. Iterating this process and using a routine chain argument, we can construct $\f{N}$ as required. The fact that $\f{N}$ has finite closures if $\f{M}$ does follows from an application of Lemma \ref{lem:PresOfClos}. \end{proof}

We now introduce the class $K_0$. It contains \textit{all} the finite structures of $\Kfin$ that \textit{may sit} strongly inside an \textit{atomic model} of $\Sfin$.

\begin{defn}\label{defn:ZeroClass}
	We let  $K_{0}=\{\f{A}:\f{A}\in\Kfin \text{ and } \delta(\f{A})=0 \}$. Further we let $\z{K_0}=\{\f{X}:\f{X}\models\Sfin^\forall$ $ \text{and for any } \f{A\finsubset{Y}} \text{ there exists } \f{B\finsubset{X}} \text{ with } \f{A\subseteq B} \text{ and }\delta(\f{B})=0   \}$.
\end{defn}

\begin{remark}
	Let $\f{D}\in\Kfin$, and $\f{X}\models\Sfin^\forall$ with $\f{D}\subseteq\f{X}$. Note that if $\delta(\f{D})=0$, then $\f{D\leq X}$. Thus it follows that if $\f{X}\in{\z{K_0}}$, then $\f{X}$ has finite closures. 
\end{remark}

We are now in a position to show that coherence of $\z\alpha$ is a sufficient condition for the existence of atomic models. 

\begin{lemma}\label{lem:ArbAtomicModels}
	Let $\z\alpha$ be coherent. Suppose $\f{M}\in\z{K_0}$ with $|M|=\kappa$. Then we can construct $\f{N}\models\Sfin$ such that $\f{N}\supseteq{\f{M}}$, $\f{N}$ is atomic and $|M|=|N|$. Thus for any $\kappa$ there is an atomic model of $\Sfin$ of size $\kappa$.
\end{lemma}

\begin{proof}
	Assume that $|M|=\kappa$. Enumerate the finite substructures of $\f{M}_0=\f{M}$ by $\{\f{B}_0,\ldots\}$. Let $\{\f{B}^{n}_0: n < \omega \}$ enumerate, up to isomorphism $\f{F}\in\Kfin$ such that $\f{B}_0\leq \f{F}$. Now consider $\f{C}_0=\icl_{\f{M}_0}(\f{B}_0)$ which is finite and has rank $0$ as $\f{M}\in\z{K_0}$. Let $\f{C}_1'=\f{C}_0\oplus_{\f{B}_0}\f{B}^{0}_0$. Since $\f{B}_0\leq \f{B}^{0}_{0}$ we have that $\f{C}_1'\in\Kfin$. As $\z\alpha$ is coherent, we can fix $\f{D}_{0}\in\Kfin$ such that $\f{C}_{1}'\subseteq{\f{D}_{0}}$ and $\delta(\f{D}_{0})=0$. Now consider $\f{M}_{1}=\f{M}_0\oplus_{\f{C}_0}\f{D}_{0}$. Note that as $\delta(\f{C}_0)=0$, $\f{C}_0\leq\f{D}_{0}$. By $(1)$ of Lemma \ref{lem:ModelsOfUniTh}, $\f{M}_1\models \Sfin^\forall$ and by $(2)$ of $\f{M}_1$ preserves closures for $\f{M}$.
	
	We claim that $\f{M_1}\in\z{K_0}$. From $(4)$ of Lemma \ref{lem:ModelsOfUniTh}  we obtain that $\f{M}_1$ has finite closures. Let ${H}={G}_1{F}_1$ be a finite substructure of $\f{M}_1$ with ${G}_1\subseteq\f{M}_0$ and ${F}_{1}\subseteq\f{D}_{1}$. Now let ${G}'=\icl_{\f{M}_0}({G}_1)$. Since $\f{M}\in\z{K_0}$, ${G}'$ is finite and $\delta(G')=\delta(\icl_{\f{M}}({G}_1))=0$. Thus it follows that $\icl_{\f{M}_1}({G}_1)={G}'$ as well. Now $\delta({G'}{D}_{1})\leq \delta({G}')+\delta({D}_{1})-e({G'}-{D}_{1},{D}_{1}-{G'})=-e({G'}-{D}_{1},{D}_{1}-{G'})\leq 0$ by using $(1)$ of Fact \ref{lem:MonoRelRankOvBase}. But as we have already established that $\f{M}_1\models\Sfin^\forall$, it follows that $\delta({G}{D}_{1})=0$. Thus any finite substructure of $\f{M}_1$ is contained in a finite substructure with rank $0$. Hence $\f{M}_1\in\z{K_0}$. 
	
	Now as noted above $\icl_{\f{M}_1}(\f{B}_0)=\f{C}_0$. Thus we may recursively form a chain $\langle \f{M}_i \rangle_{i<\omega}$ such that $\f{M}_{n+1}=\f{M}_n\oplus_{\f{C}_n}\f{D}_{n}$ so that $\delta(\f{D}_n)=0$, $\f{B}^{n}_0\subseteq\f{D}_n$, $\f{M}_{n+1}\in\z{K_0}$ and $\icl_{\f{M}_{n+1}}(\f{B}_0)=\f{C}_{n+1}=\f{C}_0$. Now consider $\f{M}^{1}=\bigcup_{i<\omega}\f{M}_n$. Now since $\f{M}_{n}\in\z{K_0}$ for each $n$, it follows immediately that $\f{M}^{1}\in{\z{K_0}}$. Note that $\f{M}^1$ satisfies all the extension formulas demanded by $\Sfin$ for $\f{B}_0$. It is clear that, by using the ideas behind the above construction of $\f{M}^{1}$ and taking unions at limit ordinals, we can build a chain $\f{M}^{\beta}\in\z{K_0}$, $\f{\beta<\kappa}$ such that each $\f{M}^{\beta}\in\z{K_0}$ and for all $\gamma<\beta$, $\f{M}^{\beta}$ contains all finite extensions of $\f{B}_{\gamma}$ needed to satisfy the extensions dictated by $\Sfin$. Now clearly $\f{M}^{\kappa}\in\z{K_0}$ and all finite substructures of $\f{M}$ have the extensions needed to satisfy the extensions dictated by $\Sfin$ \textit{in} $\f{M}^\kappa=\f{N}_0$. Now repeating this procedure we may form a $\subseteq$-chain $\langle \f{N}_\beta \rangle$ (taking unions at limit stages) where $\f{N}=\bigcup_{\beta<\kappa}\f{N}_\beta$
	satisfies $\f{N}\in\z{K_0}$ and $\f{N}\models\Sfin$. 
	
	Since there are $\f{M}\in \z{K_0}$ with $|M|=\kappa_0$ for all infinite cardinals $\kappa$ (for example, the free join over $\emptyset$ of all the elements of $K_0$ up to isomorphism, each repeated $\kappa$ many times in the free join) there are atomic models of size $\kappa$.      
\end{proof}

%\begin{remark}
%Note that in the construction of an atomic model as in Lemma \ref{lem:ArbAtomicModels}, it is important that we start with a structure $\f{M}\in \z{K_0}$.	
%\end{remark}

We now obtain the following:

\begin{thm}\label{thm:ExistenceOfAtomic}
	There exists atomic models of the theory $\Sfin$ if and only if $\z{\alpha}$ is coherent.
\end{thm}

\begin{proof}
	We begin by showing that if $\Sfin$ has atomic models, then $\z\alpha$ is coherent. To see this for each $E\in L$, fix a finite $L$ structure $\f{A}_{E}$ such that at $E$ holds on at least one subset of $\f{A}_E$ and no other relation holds on $\f{A}_E$. Let $\f{A}=\oplus_{E\in L}\f{A}_E$ be the free join of the ${\f{A}_E}$ over $\emptyset$. Let $\f{M}\models\Sfin$ be atomic with $A\subseteq M$. Thus there is some $\f{B\supseteq{A}}$ with ${B\finsubset M}$ and $\delta(\f{B})=0$. It follows that $\delta(\f{B})=0=n-\sum_{E\in L}{m_E\z\alpha_E}$. Thus $\z\alpha$ is coherent.
	
	The converse is immediate by Lemma \ref{lem:ArbAtomicModels}.  
\end{proof}

\begin{remark}\label{rmk:NonExistenceOfPrimeModShSpencer}
	The Shelah-Spencer almost sure theories do not have atomic models.  
\end{remark}

In the case that $\z\alpha$ is rational, an even stronger result than Theorem \ref{thm:ExistenceOfAtomic} is possible. In this case the models of $\Sfin$ displays similar behavior to that of classical Fra\"{i}ss\'{e} limits (i.e. theories of generics built from Fra\"{i}ss\'{e} classes where $\leq$ corresponds to $\subseteq$).   

\begin{lemma}\label{lem:Atom=ECForRW}
	Assume that $\z\alpha$ is rational. Let $\f{M}\models{\Sfin}$. Now $\f{M}$ is atomic if and only if $\f{M}$ is an e.c. model. Hence every model of $\Sfin$ embeds isomorphically into an atomic model of $\Sfin$. 
\end{lemma}

\begin{proof}
	Assume that $\z\alpha$ is rational and that $\f{M}\models\Sfin$ and note that if $\z\alpha_{E}$ is rational, then $\Sfin$ has finite closures. By Theorem \ref{thm:ClassAtModMain} we immediately obtain that $\f{M}$ is atomic if and only if $\f{M}$ is an e.c. model. Further, well known results about $\forall\exists$-theories tell us that every model of $\Sfin$ sits as a \textit{substructure} of an e.c. model of $\Sfin$ of the same cardinality. Since \textit{any} model of $\Sfin$ has finite closures, it follows that every model embeds \textit{isomorphically} into an atomic model.
\end{proof}

\begin{remark}\label{rmk:ExtensionsOfUniTh}
	Assume that $\z\alpha$ is rational. It is easily seen that any $\f{X}\models\Sfin^\forall$ has finite closures. Thus it follows from Lemma \ref{lem:CompletionsOfUniTh1} that any $\f{X}\models\Sfin^\forall$ embeds isomorphically into some $\f{N}\models\Sfin$ (taking the free join of $\aleph_0$ many non-isomorphic copies of $\f{X}$ over $\emptyset$ if $\f{X}$ is finite). Thus from Lemma \ref{lem:Atom=ECForRW}, it follows that given any $\f{X}\models\Sfin^\forall$, we see that there is an atomic $\f{N}'\models\Sfin$ such that $\f{X}$ embeds isomorphically into $\f{N}'$        
\end{remark}

We will now explore the behavior of atomic models when $\z\alpha$ is coherent but $\z\alpha$ is not  rational. We begin by showing that \textit{any} countable $\f{X}\models\Sfin^\forall$ with finite closures embeds isomorphically into the countable atomic model of $\Sfin$ mimicking the behavior of Remark \ref{rmk:ExtensionsOfUniTh}. Recall that if $\f{X}\in \z{K_0}$, then $\f{X}$ has finite closures.

\begin{lemma}\label{lem:SlayingTheHydra}
	Let $\z\alpha$ be coherent and let $\f{M}\models\Sfin^\forall$ be countable with finite closures. Then
	\begin{enumerate}
		\item There exists a countable $\f{M}^*\in\z{K_0}$ with $\f{M^*}\supseteq{\f{M}}$.
		\item There exists a countable atomic $\f{N}\models\Sfin$ such that $\f{M\subseteq N}$. 
	\end{enumerate}
\end{lemma}

\begin{proof}
	$(1)$: Since $\f{M}$ has finite closures, we may write $\f{M}=\bigcup_{i<\omega}\f{A}_i$ where $\f{A}_i\leq\f{A}_{i+1}$ for each $i<\omega$. We will now construct $\f{M^*}$ as the union of a countable $\subseteq$-chain $\f{M}_0\subseteq\f{M}_1\subseteq...$ with $\f{M}=\f{M}_0$ and $|M_{n}-M_0|$ finite for all $n<\omega$ as follows: Let $\f{M}_0=\f{M}$ and given $\f{M}_n$, let $\f{A}_n^*=\icl_{\f{M}_n}(A_n\cup{(M_n-M_0)})$. Using Theorem \ref{thm:AnnhilConstruc} choose $\f{B}_n\in\Kfin$ with $\f{A}_n^*\subseteq \f{B}_n^*$ and $\delta(\f{B}_n)=0$. Let $\f{M}_{n+1}=\f{M}_n\oplus_{\f{A}_n^*}\f{B}_n$. As $\f{A}_n^*\leq \f{M^*}$, it follows from Lemma \ref{lem:ModelsOfUniTh} that each $\f{M}_n\models\Sfin^\forall$. Clearly $|{M}_n-M_0|$ is finite as claimed. As each $\f{M}_n\models\Sfin^\forall$, $\f{M^*}\models\Sfin^\forall$ where $\f{M^*}=\bigcup_{i<\omega}\f{M}_n$. Note that given \textit{any} finite set of $\f{A}\subseteq \f{M^*}$, there is some $n<\omega$ such that $\f{A}\subseteq \f{M}_{n}$. By construction, it follows that there is some $k<\omega$ such that $\f{A}\subseteq\f{B}\subseteq\f{M}_{n+k}$ with $\f{B}$ finite and $\delta(\f{B})=0$. Thus it follows that $\f{M^*}\in\z{K_0}$.\\ 
	
	\noindent$(2)$: We now do an alternating chain argument: We let $\f{M}^*_0=\f{M}$. Thus $\f{M}_0^*$ has finite closures. We build $\f{M}^*_{2n+1}\models\Sfin$ with $\f{M}^*_{2n}\subseteq\f{M}^*_{2n+1}$ such that $\f{M}^*_{2n+1}$ has finite closures, preserves closures for $\f{M}^*_{2n}$ and is countable by use of Lemma \ref{lem:CompletionsOfUniTh1}. We let $\f{M}^*_{2n+2}$ be such that $\f{M}^*_{2n+1}\subseteq\f{M}^*_{2n+2}$ and $\f{M}^*_{2n+2}\in\z{K_0}$ which exists by use of $(1)$. We let $\f{N}=\bigcup_{n<\omega}\f{M}^*_{n}$. Let $\f{B}\finsubset\f{N}$. Now  as $\f{B}\subseteq \f{M}^*_{2n_0+1}$ for some $n_0$, a routine argument shows that $\f{N}\models{\Sfin}$. As $\f{B}\subseteq \f{M}^*_{2n_0+1} \subseteq \f{M}^*_{2n_0+2}$ it follows that  $\f{D} = icl_{\f{M}_{2n_0+2}}(\f{B})$ is finite and $\delta(\f{D})=0$. Thus it follows that $icl_{\f{M}_{2n_0+2}}(\f{B})=icl_{\f{N}}(\f{B})$ and hence  $\f{N}\in\z{K_0}$. Thus $\f{N}$ is (up to isomorphism), the unique countable atomic model of $\Sfin$ by Theorem \ref{thm:ClassAtModMain}.
\end{proof} 

We now proceed to show that this behavior may fail for arbitrary $\f{X}\models\Sfin^\forall$. 

\begin{defn}\label{defn:tent-like}
	Call a structure  $\f{N}\models\Sfin^\forall$ \textit{tent-like} over $\f{M}$ if \begin{enumerate}
		\item  $M$ is a set of points with no relations between them 
		\item  For all pairs $\{a,b\}$ of distinct elements from $M$, there is a unique minimal pair $
		(\{a,b\}, \f{F}_{a,b})$ in $\f{N}$. 
		\item  $N= \bigcup_{a,b\in{A}, a\neq b} F_{(a,b)}$  
		\begin{enumerate}
			\item For distinct $a,b,b'\in{M}$, $\f{F}_{a,b}, \f{F}_{a,b'}$ are freely joined over $a$
			\item For distinct $a,a',b,b'\in{M}$,  $\f{F}_{a,b}, \f{F}_{a,b'}$ are freely joined over $\emptyset$
		\end{enumerate}
		\item $\icl_{\f{N}}(\{a\})=\{a\}$ for each $a\in M$
	\end{enumerate}
	We will refer to $M$ as the \textit{base} of the tent $\f{N}$ over $\f{M}$.
\end{defn}

\begin{remark}\label{rmk:tentsHaveFiniteClosures}
	Note that given a finite subset $A_0 = \{a_{n_1},\ldots,a_{n_k}\}$ of $M$ we have that ${A}'=\bigcup \f{F}_{a,b}\subseteq{icl_{\f{N}}({A}_0)}$ where $(a,b)$ ranges through distinct pairs from $A_0$. We claim that this set is closed. Assume to the contrary that there is a minimal pair $({D},{DG})$ where ${D}\subseteq{{A}'}$ and $G$ is disjoint from $A'$. Note $\delta({G/D})\geq{\delta({G/A'})}$ using $(2)$ of Fact \ref{lem:MonoRelRankOvBase}. Since $\f{N}$ is tent-like over $\f{M}$, $\delta({G/A'})=\delta({G/A_0})$. From the tent-likeness of $\f{N}$ over $\f{M}$ and our choice of $A'$ and $G$, it follows that $\delta({G/A_0}) = \sum_{(a,b)\notin {A_0\times A_0}, a\neq b} \delta({G}\cap{F}_{a,b}/{A}_0)$. Thus 
	$\delta({G}\cap{F}_{a,b}/{A}_0)$ for $(a,b)\notin {A_0\times A_0}, a\neq b$ reduces to either $\delta({G}\cap{F}_{a,b})$ or $\delta({G}\cap{F}_{a,b}/c)$ where $c=a$ or $c=b$. But since each $a'\in{A}$ is its own closure in $\f{N}$ it follows that the $\delta({G}\cap{F}_{a,b}/c)\geq{0}$. Thus it follows that ${A}'$ is closed. Now by noting that each finite subset lies in finitely many of the ${F}_{a,b}$ it follows that $\f{N}$ has finite closures.    %$\delta({G/A'})=\delta({G})-e({G},{A}')=\delta({G})-e({G},{A}_0)= \sum \delta({G}\cap{F}_{a,b}/{A}_0) \geq{0}$
\end{remark}

\begin{remark}\label{rmk:TentsNotInKZero}
	Note that if $\f{N}\models\Sfin^\forall$ is tent like over $\f{M}$, then $\f{N}\notin\z{K_0}$ as $\delta(\icl(a))=1$ for each $a\in M$.
\end{remark}

\begin{lemma}\label{lem:constTent}
	Let $\z\alpha$ be coherent but not rational. Suppose $\f{N}\models\Sfin^\forall$ tent-like over $\f{M}$ where $M$ is countable.  Then there is an extension $\f{N}^*$ of $\f{N}$ over $\f{M}^*$ where $\f{M}\subseteq \f{M}^*$ and $M^*$ has universe $M\{a^*\}$, where $a^*$ is a single new point such that $\f{N}^*$ is tent-like over $\f{M}^*$. Thus there is some $\f{N'}$ where the corresponding base $\f{M'}$ has $|M'|=\aleph_1$.   
\end{lemma}

\begin{proof}
	Enumerate $M=\{a_n:n\in\omega\}$. Fix $E\in L$ such that $\z\alpha_E$ is irrational. Now for each $n\in \omega$ we may choose an essential minimal pair $\f{F}_{(a_n,a^*)}$ over $\{a_n,a^*\}$ such that $-1/2^{n+1} < \delta(\f{F}_{(a_n,a^*)}/\{a_n,a^*\}) < 0$ using Theorem \ref{lem:InfMinPa}. Let ${D'\subseteq F_{a_n,a^*}}$. Now if $D'\cap{\{a_n,a^*\}}$ contains exactly one element, then $\delta(D'/D\cap\{a_n,a^*\})\geq 0$. So suppose that $D'\cap{\{a_n,a^*\}}={\{a_n,a^*\}}$. Since $\delta(\{a_n,a^*\}/\{a^*\})=\delta(\{a_n,a^*\}/\{a_n\})=1$ and $\delta({D'/\{c\}})= \delta({D'/\{a_n,a^*\}})+\delta(\{a_n,a^*\}/c) \geq -1/2^{n+1}+1\geq 0$ where $c=a_n$ or $c=a^*$ it follows that $\{a_n\},\{a^*\}\leq \f{F}_{a_n,a^*}$. Now consider the structure $\f{N}^*$ with universe $N\cup\{a^*\}\cup\bigcup_{a_n\in{A}}{F_{a^*,a_n}}$ with

	\begin{enumerate}
		\item For distinct $a,b,b'\in{Ma^*}$, $\f{F}_{a,b}, \f{F}_{a,b'}$ are freely joined over $a$
		\item For distinct $a,a',b,b'\in{Ma^*}$,  $\f{F}_{a,b}, \f{F}_{a,b'}$ are freely joined over $\emptyset$
	\end{enumerate}

	Clearly  $M\{a^*\}$ is a set of points with no relations between them. Note that we have shown that $\{a^*\},\{a_n\}\leq \f{F}_{a_n,a^*}$. Let $\f{G\finsubset{N^*}}$. Suppose that the $G\cap{M\{a^*\}}=\emptyset$. Then because of the conditions regarding free joins we see that $\delta(\f{G})=\sum \delta(\f{F}_{a,b}\cap\f{G})\geq{0}$. Now consider the case $G\cap{M\{a^*\}}\neq\emptyset$. Put $G'=G\cap{A\{a^*\}}$. Now $\delta(\f{G/G'})=\delta(\f{(N\cap{G})/G'})+\delta(\f{(N^*-N)\cap{G}/G'})=\delta(\f{(N\cap{G})/\f{N\cap{G'}}})+\delta(\f{(N^*-N)\cap{G}/G'})$ where $\delta(\f{(N\cap{G})/G'})=\delta(\f{(N\cap{G})/\f{N\cap{G'}}})$  follows by considering the fact that the underlying finite structures are freely joined. Now $\delta(\f{G})=\delta(\f{G'})+\delta(\f{(N\cap{G})/\f{N\cap{G'}}})+\delta(\f{(N^*-N)\cap{G}/G')}$. Suppose that $a^*\notin{G'}$. Then $\delta(\f{G'})+\delta(\f{(N\cap{G})/\f{N\cap{G'}}})=\delta(\f{(N\cap{G})})$ and $\delta(\f{(N^*-N)\cap{G}/G'})\geq{0}$ by using an argument similar to that in Remark \ref{rmk:tentsHaveFiniteClosures}. So assume that $a^*\in{G'}$. Now $\delta(\f{G})=\delta(\f{G'\cap{N}})+\delta(\f{(N\cap{G})/\f{N\cap{G'}}})+\delta(a^*)+\delta(\f{(N^*-N)\cap{G}/G')}$. It follows that $\delta(\f{G'\cap{N}})+ \delta(\f{(N\cap{G})/\f{N\cap{G'}}})\geq{0}$ by an argument similar to the above. But by construction of the new minimal pairs $\delta(a^*)+\delta(\f{(N^*-N)\cap{G}/G')}\geq{1-\sum 1/{2^{n+1}}}\geq{0}$. Thus $\f{N^*}\models\Sfin^\forall$.  
	
	Now each pair of points $\{a,b\}$ from $M\{a^*\}$ has a minimal pair over it; i.e. $(ab,\f{F}_{a,b})$ is a minimal pair. Now consider  $\f{F}_{a,b}$. Note that since $a\leq\f{F}_{a,b'}$  and $b\leq\f{F}_{b,b'}$ and using the various properties regarding how the $\f{F}_{c,d}$ are freely joined and arguing in a similar manner to Remark \ref{rmk:tentsHaveFiniteClosures} yields that $\f{F}_{a,b}$ is closed in $\f{N^*}$ which establishes that there is a unique  minimal pair over $ab$. Now it also follows that for any $a\in{M\{a^*\}}$ the closure of $a$ is itself. Thus $\f{N^*}$ is also tent-like.
	
	By iterating this  $\omega_1$ many times we obtain a tent-like structure where the corresponding $\f{N}'$ over $\f{M}'$ where $|M'|=\aleph_1$. 
\end{proof}

\begin{lemma}\label{lem:CoherenceWithIrrational}
	Let $\z\alpha$ be coherent but not rational. Then there is $\f{X}\models\Sfin^\forall$ of size $\aleph_1$ such that $\f{X}$ has finite closures but there is no atomic model $\f{N}$ of $\Sfin$ such that $\f{N\supseteq{X}}$. Thus there is $\f{M}\models\Sfin$ such that $\f{M}$ does not embed isomorphically into any atomic model of $\Sfin$. 
\end{lemma}

\begin{proof}
	Let $\f{X}\models\Sfin^\forall$ be tent-like over $\f{Y}$ where $Y=\{a_i : i<\omega_1 \}$. We claim that there is no $\f{N\supseteq{X}}$ such that $\f{N}$ is an atomic model of $\Sfin$. 
	
	Assume to the contrary that there is such a $\f{N}$. Now for any $a_{\beta}$, $\icl_{\f{N}}(\{a_\beta\})$ would be finite and $\delta(\icl_{\f{N}}(\{a_\beta\}))=0$ by use of Theorem \ref{thm:ClassAtModMain}. Note that for $\beta, \gamma$ distinct, $\f{F}_{\beta,\gamma}\subseteq\icl_{\f{N}}(\{a_\beta, a_\gamma\} )$. Now either $(\f{F}_{\beta,\gamma}-\{a_\beta, a_\gamma\})\cap\icl_{\f{N}}(\{a_\beta\})  \neq \emptyset$ or $(\f{F}_{\beta,\gamma}-\{a_\beta a_\gamma \})\cap \icl_{\f{N}}(\{a_\beta\}) \neq \emptyset$. For if not \[
	\begin{array}{lcl}
	\delta(\icl_{\f{N}}(\{a_\beta\})\icl_{\f{N}}(\{a_\gamma\})) & = & \delta(\icl_{\f{N}}(\{a_\beta\} ))+\delta(\icl_{\f{N}}(\{a_\gamma\}))-\delta(\icl_{\f{N}}(\{a_\beta\}) \cap\icl_{\f{N}}(\{a_\gamma\} ))\\
	
	& & -e(\icl_{\f{N}}(\{a_\beta\})-\icl_{\f{N}}(\{a_\gamma\}), \icl_{\f{N}}(\{a_\gamma\} )-\icl_{\f{N}}(\{a_\beta\} ))
	\end{array}
	\] 
	
	by use of $(1)$ of Fact \ref{lem:BasicDel1}). This implies that $\delta(\icl_{\f{N}}(\{a_\beta\} )\icl_{\f{N}}(\{a_\gamma\}))=0$. But then $\icl_{\f{N}}(\{a_\beta\})\icl_{\f{N}}(\{a_\gamma\})$ is closed. Thus we obtain that, $(\f{F}_{\{a_\beta,a_\gamma\}}-\{a_\beta a_\gamma\})\subsetneq\icl_{\f{N}}(\{a_\beta a_\gamma\})\subseteq{\icl_{\f{N}}(\{a_\beta\})\icl_{\f{N}}(\{a_\gamma\})}$, a contradiction.  
	
	Now for each $\beta$, $\icl_{\f{N}}(\{a_\beta\})$ is finite. Thus there is some $\beta^*>\beta$ such that \[\icl_{\f{N}}(\{a_\beta\})\cap{(\f{F}_{\{a_\beta,a_\gamma\}}-\{a_\beta, a_\gamma\})}=\emptyset\] for all $\gamma>\beta^*$. But now by doing a standard catch your tail argument,  we can find $\beta'<\omega_1$ such that for all $\beta<\beta'$, if $\icl_{\f{N}}(\{a_\beta\})\cap{(\f{F}_{\{a_\beta,a_\gamma\}}-\{a_\beta, a_\gamma\})}\neq\emptyset$, then $\gamma<\beta'$. Choose $\gamma>\beta'$. For all $\beta<\beta'$, $\icl_{\f{N}}(\{a_\beta\})\cap{(\f{F}_{\{a_\beta,a_\gamma\}}-\{a_\beta, a_\gamma\})}=\emptyset$. Hence  $\icl_{\f{N}}(\{a_\gamma\})\cap{(\f{F}_{\{a_\beta,a_\gamma\}}-\{a_\beta, a_\gamma\})}\neq\emptyset$. But this is contradictory as $\icl_{\f{N}}(\{a_\gamma\})$ is finite and the $\f{F}_{\{a_\beta,a_\gamma\}}-\{a_\beta a_\gamma\}$ are distinct non-empty sets.
	
	We can do an easy chain argument argument to show that there is some $\f{X\subseteq M}$ and $\f{M}\models\Sfin$. Clearly no such $\f{M}$ embeds into an atomic model as otherwise, $\f{X}$ would to. This finishes the proof.
\end{proof}

We finally finish with Theorem \ref{thm:RationalAlphaAndCoherence}, which shows that when $\z\alpha$ is coherent, $\z\alpha_E$ being rational for all $E\in L$ can be characterized in terms of isomorphic embeddability into atomic models.   

\begin{thm}\label{thm:RationalAlphaAndCoherence}
	Let $\z\alpha$ be coherent. The following are equivalent \begin{enumerate}
		\item $\z\alpha$ is rational
		\item Every $\f{M}\models\Sfin$ embeds isomorphically into an atomic model of $\Sfin$
	\end{enumerate}	
\end{thm}

\begin{proof}
	The proof of this statement is immediate from Lemma \ref{lem:Atom=ECForRW} and Lemma \ref{lem:CoherenceWithIrrational}.
\end{proof}

\begin{remark}\label{rmk:GraphLifeWithWeightOne}
	Theorem \ref{thm:OmitPrelim2} fails for the case that for each $E$ in $L$, arity of $E$ is $2$ and  $\z\alpha(E)=1$. However it is possible to obtain the quantifier elimination result and the results about atomic models by proving Theorem \ref{thm:AnnhilConstruc} and Lemma \ref{lem:Omit1}. These results require ad hoc constructions that apply only to this specific case. We have omitted this case in favor of a more streamlined presentation of the results. 
\end{remark}

\appendix 

\section{Some Relevant Number Theoretic Facts}\label{App:NumTh}

The number theoretic results concerning Diophantine equations can be found in Chapter 5 of \cite{IZH} and the number theoretic results concerning continued fractions can be found in Chapter 7 therein.

\begin{remark}\label{rmk:InfSolDioEqn}
	We note that in the case all the $\z\alpha_E$ are rational the equation $n-{\sum_{E\in L}{\z\alpha_Em_E}}=-\frac{1}{c}$ has infinitely many positive integer solutions, i.e. solutions where $n$ and all of the $m_E$ are positive. To see this, note that we can relate the solutions of such an equation to the solutions of a linear Diophantine equation: Given a fixed $E'\in{L}$, let $\prod_{\hat{E'}}(q_E)$ be the product of the $q_{E}$ for $E\in{L-\{E'\}}$.  Now using this fact that $c=\frac{\prod_{E\in L}{q_E}}{\gcd_{L\in E}{q_E}}$ and multiplying through by $\prod{q_E}$ we obtain the equation:  $n\prod_{E\in L}{q_E} - \sum_{E\in L}(m_Ep_E\prod_{\hat{E}}(q_{E^*})) = -\gcd_{E\in L}(q_E)$.
	
	We can now show that $\gcd_{E\in L}(\prod{q_E}, p_E\prod_{\hat{E}}(q_{E'}))=\gcd_{E\in L}(q_E)$ by repeated use of $\gcd(ab,af)=a\gcd(b,f)$, $\gcd(a_1,\ldots,a_n)=\gcd(\gcd(a_1,a_2),\ldots{a_n})$ and the fact that for relatively prime $\gcd(b,af)=\gcd(b,a)$. Thus the above equation has infinitely many solutions. In particular it has infinitely many solutions $\langle n,m_E\rangle_{E\in L}$ where $n,m_E>0$ for all $E\in L$.
\end{remark}

\begin{remark}\label{rmk:ContinuedFractions}
	Let $0 < \beta < 1$ be irrational. Note that $\beta$ has a simple continued fraction form $[0:a_1,a_2,\ldots]=0+1\frac{1}{a_1+\frac{1}{a_2+\cdots}}$ where $a_i\in\omega$ is positive for $i\geq 1$. Let $p_k/q_k=[0:a_1,\ldots,a_k]$ be the simple continued fraction approximation restricted to $k$-terms. Now: 
	\begin{enumerate}
		\item $p_k, q_k$ are increasing sequences (and hence $p_k,q_k\rightarrow\infty$)
		\item $\langle p_{2k}/q_{2k} : k\in\omega \rangle$ is a strictly increasing sequence that converges to $\beta$
		\item For even $k$, $\frac{1}{q_{k}(q_{k}+q_{k+1})} < \beta -\frac{p_k}{q_k} < \frac{1}{q_k q_{k+1}} $
	\end{enumerate}
	
	Now it follows that $ -\frac{1}{q_{2k}} < p_{2k}-q_{2k}\beta < -\frac{1}{q_{2k}+q_{2k+1}}$. This easily yields that $\lim_{k}p_{2k}-q_{2k}\beta=0$.
\end{remark}

\bibliographystyle{plain}
\bibliography{AtomMod}
\Addresses
\end{document}